\documentclass[12pt, reqno]{amsart}
\usepackage[margin=1in]{geometry}

\usepackage{amssymb}
\usepackage{amsmath}
\usepackage{amsthm}
\usepackage{mathtools}
\usepackage{hyperref}
\usepackage{mathrsfs}
\usepackage{amsrefs}
\usepackage{bbm}
\usepackage{bm}
\usepackage{appendix}
\usepackage{amsfonts,mathtools}
\usepackage{bbm}
\usepackage{comment}
\usepackage{enumitem}

\usepackage{amsmath}
\usepackage{flexisym}
\usepackage{breqn}
\usepackage{hyperref}

\usepackage{amsthm,enumitem, physics, tikz}

\usepackage{dutchcal}

\DeclareMathAlphabet{\dc}{U}{dutchcal}{m}{n}

\DeclareMathAlphabet{\mathcal}{OMS}{cmsy}{m}{n}

\renewcommand{\Re}{\mathrm{Re}}

\renewcommand{\le}{\leq}

\makeatletter
\renewcommand{\pmod}[1]{\allowbreak\mkern7mu({\operator@font mod}\,\,#1)}
\makeatother

\newcommand{\mc}{\mathcal}
\newcommand{\ba}{\mathbf{a}}
\newcommand{\ms}{\mathscr}
\newcommand{\eps}{\varepsilon}
\newcommand{\Li}{\mathrm{Li}}
\newcommand{\mf}{\mathfrak}
\newcommand{\rad}{\mathrm{rad}}
\newcommand{\E}{\mathbb{E}}   
\newcommand{\PR}{\mathbb{P}}  
\newcommand{\be}{\begin{equation}}
\newcommand{\ee}{\end{equation}}

\newcommand{\PP}{\mathcal{P}}
\newcommand{\QQ}{\mathcal{Q}}
\newcommand{\cS}{\mathcal{S}}

\newcommand{\pfrac}[2]{\left(\frac{#1}{#2}\right)}  

\newcommand{\disc}{\mathrm{disc}}

\numberwithin{equation}{section}
\newcommand{\Q}{\mathbb{Q}}
\newcommand{\R}{\mathbb{R}}
\newcommand{\N}{\mathbb{N}}

\newcommand{\Z}{\mathbb{Z}}
\newcommand{\ideal}{\vartriangleleft}

\usepackage{amsthm}

\newtheorem{theorem}{Theorem}[section]
\newtheorem{lemma}[theorem]{Lemma}
\newtheorem{definition}[theorem]{Definition}
\newtheorem{corollary}[theorem]{Corollary}
\newtheorem{proposition}[theorem]{Proposition}

\title{Prime-free discs in imaginary quadratic fields}
\author{Tanmay Khale}
\address{Department of Mathematics, University of Illinois, Urbana, IL 61801, USA}
\email{tnkhale2@illinois.edu}
\begin{document}

\begin{abstract}
Suppose $K$ is an imaginary quadratic field, and let $N_K$ denote the field norm in $\mc{O}_K$. Let  $B(\dc{x}_0,r) = \{\dc{x} \in \mc{O}_K: |N_K(\dc{x}-\dc{x}_0)| < r\}$. Let 
$G_K(X) = \max \{r > 0: \text{there exists } \dc{x}_0 \in \mc{O}_K \text{ such that } |N_K(\dc{x}_0)| \leq X \text{ and } B(\dc{x}_0,r) \text{ contains no primes} \}$. We show that $ G_{K}(X) \gg_K (\log X) \frac{\log_2(X) \log_4(X)}{\log_3	(X)}	 $.
\end{abstract}

\maketitle
{\centering\footnotesize Dedicated to the memory of Zachary H. Polansky.\par}

\section{Introduction}\label{section-introduction}
Suppose $K$ is a number field, and let $N_K$ denote the field norm in $\mc{O}_K$. Let  $B(\dc{x}_0,r) = \{\dc{x} \in \mc{O}_K: |N_K(\dc{x}-\dc{x}_0)| < r\}$. Let $G_K(X)$ denote the size of the largest ``hole" in primes of norm at most $X$. That is,
\begin{align*}
G_K(X) &= \max \{r > 0: \text{there exists } \dc{x}_0 \in \mc{O}_K \text{ such that } |N_K(\dc{x}_0)| \leq X \\
&\qquad \text{ and } B(\dc{x}_0,r) \text{ contains no primes} \}	.
\end{align*}
By the prime ideal theorem (Theorem \ref{landauprimeideal} below), $G_K(X)$ is at least $(1+o_K(1)) \log(X)$ (where $o_K(1)$ denotes a function depending on $K$ which tends to zero as $X \to \infty$). For $K=\Q$, Westzynthius, Erd\H{o}s and Rankin successively improved the lower bound above, showing for a \textit{fixed} constant $c > 0$, and writing $\log_k(x)$ to denote the $k$-fold iterated logarithm, that
\begin{equation*}
G_{\Q}(X) \geq (c + o(1)) (\log X)  \frac{\log_2(X) \log_4(X)}{(\log_3(X))^2} .
\end{equation*}
The above stood as the best-known result for 76 years, until in 2014 two papers \cites{4author, 1author} independently proved that the constant $c$ above could be taken to be \textit{arbitrarily} large. In a subsequent collaboration \cite{fgkmt}, the authors of the two papers showed that
\begin{equation}\label{fgkmt-result}
G_{\Q}(X) \gg (\log X) \frac{\log_2(X) \log_4(X)}{\log_3	(X)}	.
\end{equation}
In the more general case where $K$ is any imaginary quadratic field, the trivial lower bound $(1+o_K(1)) \log(X)$ has not previously been improved. The objective of this paper is to prove a lower bound generalizing \eqref{fgkmt-result} to any imaginary quadratic field $K$. Our main result is the following:
\begin{theorem}\label{mainthm}
Let $K$ be an imaginary quadratic field. Then, we have
\begin{align}\label{size}
G_{K}(X) \gg_K (\log X) \frac{\log_2(X) \log_4(X)}{\log_3	(X)}	.
\end{align}
\end{theorem}
\subsection{Organization} In Section \ref{section-rankin}, we use the Chinese Remainder theorem alongside estimates for smooth algebraic integers to reduce Theorem \ref{mainthm} to Theorem \ref{sieved}. In Section \ref{section-page-bv} we utilize the Landau--Page theorem for number fields to obtain a version of the Bombieri--Vinogradov theorem for number fields with a strong error term, for use in Section \ref{section-sieve-weights}. The bulk of this paper, in Section \ref{section-sieve-weights}, is devoted to the proof of Theorem \ref{prop7.14}, a number field variant of the uniform estimates for prime $k$-tuples in \cite{denseclusters}. We use this to deduce Theorem \ref{thm8.6}, which gives the existence of a sieve weight analogous to sieve weight defined in \cite{fgkmt}*{Section~7}. In Section \ref{section-probability-weights}, we define the probability weight used to prove Theorem \ref{sieved}, and using Theorem \ref{thm8.6} we deduce Corollary \ref{immediate}. Finally, combining Corollary \ref{immediate} with Theorem \ref{packing-quant-cor} (which is a corollary of the hypergraph covering theorem in \cite{fgkmt}*{Theorem 3}), we deduce Theorem \ref{sieved}.

\subsection{Acknowledgements}
The author thanks Kevin Ford, Jesse Thorner, and Gergely Harcos for many helpful comments and corrections.

\section{Notational conventions}\label{section-notation}
 Throughout this paper we adopt the following typographical convention:
  \begin{enumerate}
    \item \emph{Ideals} of the ring of integers $\mathcal O_K$ are denoted by
          fraktur letters, e.g.,\ $\mathfrak a,\mathfrak p\subset\mathcal O_K$.
    \item \emph{Algebraic integers} (elements of $\mathcal O_K$) are written in
          Dutch calligraphic letters, e.g.\ $\dc a,\dc b\in\mathcal
          O_K$.
    \item \emph{Rational integers} (elements of $\Z$) are denoted by the default TeX math font, e.g., $n, p, q \in \Z$. 
  \end{enumerate}
The implied constants in this paper may depend on the imaginary quadratic field $K$ in an unspecified manner. We write $f = O_\leq(g)$ if $|f| \leq g$. We write $\mf{a} \ideal \mc{O}_K$ to mean that $\mf{a}$ is an integral (i.e., not fractional) ideal of $\mc{O}_K$.
Define
\begin{equation*}
\pi_G(x) = \#\{\mathfrak{p} \ideal \mc{O}_K: \mathfrak{p} \text{ prime, } N_{K}(\mathfrak{p}) \leq x\}.
\end{equation*}
For an ideal $\mf{a} \ideal \mc{O}_K$, let $N_K(\mf{a})$ denote the ideal norm $N_{K/\Q}(\mf{a})$ of the ideal $\mf{a}$. As usual, we define $N_K$ for algebraic integers $a \in \mc{O}_K$ by $N_K(a) = N_K((a))$. For $\mf{a} \ideal \mc{O}_K$, define
\[
\Lambda(\mf{a})
=
\begin{cases}
\log(N_K(\mf{p}))	&  \mf{a} = \mf{p}^k \\
0 & \text{otherwise.}
\end{cases}
\]
We define the M\"obius function $\mu$ on prime power ideals by $\mu(\mf{p}) = -1$ and $\mu(\mf{p}^k) = 0$ for $k \geq 2$. We extend $\mu$ to all prime ideals $\mf{a} \ideal \mc{O}_K$ multiplicatively.

For $n \in \Z$, define
\[
\rad(n) = \prod_{p \mid n} p.
\]
For $n \in \Z$, we write $P^+(n)$ and $P^-(n)$ to denote the largest and smallest prime factors of $n$ respectively, with the conventions that $P^{+} (1) = 1$ and $P^{-}(1) = \infty$.

Whenever we use the variable $\mf{q} \ideal \mc{O}_K$, we assume that $\mf{q}$ is relatively prime to the difference between any two units in $\mc{O}_K$, which excludes only $O(1)$ choices of $\mf{q}$. We also assume that the units do not represent all reduced residue classes modulo $\mf{q}$.

Finally, we write $\sum'$ to denote a sum over ideals $\mf{q}$ composed of non-ramifying prime ideals. 
\section{Preliminaries}\label{section-preliminaries}
In this section, we record several standard results for later use. First, we require Landau's prime ideal theorem, in the following form (from \cite{mv}*{Theorem 8.9}):
\begin{theorem}[Landau]\label{landauprimeideal}
Let $K$ be an algebraic number field of finite degree over $\mathbb{Q}$, and let $\mathcal{O}_K$ denote the ring of algebraic integers in $K$.  Then for $x \ge 2$, the number of prime ideals $\mathfrak{p}$ of $\mathcal{O}_K$ with
\[
N_K(\mathfrak{p}) \le x
\]
is
\[
\#\{\mathfrak{p}\ideal\mathcal{O}_K : N_K(\mathfrak{p})\le x\}
\;=\;\mathrm{Li}(x)\;+\;O_K\bigl(x\exp(-c\sqrt{\log x})\bigr),
\]
where $c>0$ is a constant depending on $K$.
\end{theorem}
Second, we require the following consequences of the Chebotarev density theorem in \cite{lo}*{Theorem 1.3}:
\begin{theorem}\label{cheb}
Let $K$ be an imaginary quadratic field. For $p \in \Z$, we say $p$ splits in $\mc{O}_K$ if $(p) = \mf{p}_1 \mf{p}_2$ for prime ideals $\mf{p}_1, \mf{p}_2 \ideal \mc{O}_K$, and we say that $p$ is inert if $(p) \ideal \mc{O}_K$ is prime. Then, for constants $C_1, C_2, c$, depending on $K$, we have the following:
\begin{align*}
&\sum_{\substack{p \leq x \\ p~\mathrm{ splits}}} \frac{\log p}{p} = \frac{1}{2}\log(x) + C_1 +  O_K\left(\exp(-c\sqrt{\log x})\right), \\
 &\sum_{\substack{p \leq x \\ p~\mathrm{ inert}}} \frac{\log p}{p} = \frac{1}{2}\log(x) + C_2 + O_K\left(\exp(-c\sqrt{\log x})\right).
\end{align*}
\end{theorem}
\section{Ray classes in $K$ and Hecke $L$-functions}\label{section-ray-hecke}
In this section, we define the notion of ray classes in the number field $K$ (which generalize arithmetic progressions over the integers), and Hecke $L$-functions (which generalize Dirichlet $L$-functions over $\Q$). Proofs of the various assertions in this section can be found in \cite{neukirch}*{Chapters 6 and 8}.

Let $J^{\mathfrak{q}}$ be the set of fractional ideals coprime to $\mf{q}$, and let $P^{\mf{q}}$ denote the set of principal fractional ideals $(\dc{a})$ such that there exist $\dc{b},\dc{c} \in \mc{O}_K$ with $\dc{b} \equiv \dc{c} \equiv 1\pmod{\mf{q}}$ and $(\dc{a}) = (\dc{b})(\dc{c})^{-1}$. Then, $H^{\mf{q}} := J^{\mf{q}}/P^{\mf{q}}$ is called the \textit{ray class group} modulo $\mf{q}$.

Let $J_1^{\mf{q}}$ be the set of principal fractional ideals coprime to $\mf{q}$. Then, $J_1^{\mf{q}}$ is in one-to-one correspondence with the set 
\[
\{\{\dc{u} \dc{a}: u \in \mathcal{O}_K^{\times}\}: \dc{a} \in \mc{O}_K, ~ (\dc{a},\mathfrak{q})=1\}.
\]

We will write $\mf{a} \equiv \mf{b}\pmod{\mf{q}}$ if $\mf{a}$ and $\mf{b}$ represent the same equivalence class in the ray class group $H^{\mf{q}}$. Similarly, for $\dc{a} \in \mc{O}_K$, we write $\mf{b} \equiv \dc{a} \pmod{\mf{q}}$ if $\mf{b}$ is principal and there exists a generator $\dc{b}$ of $\mathfrak{b}$ such that $\dc{b} \equiv a \pmod{\mf{q}}$ (in other words, when $\mf{b}$ and $(a)$ represent the same equivalence class in the ray class group $H^{(q)}$). Put yet another way, for a principal ideal $\mf{b} = (\dc{b})$, we write $\mf{b} \equiv a \pmod{q}$ if there exists a unit $\dc{u} \in K$ such that $\dc{b} \equiv \dc{u}\dc{a} \pmod{\mf{q}}$. 

For any ideal $\mf{q}$ of $\mc{O}_K$, let $h(\mf{q}) = |H^{\mf{q}}|$ denote the size of the ray class group modulo $\mf{q}$. 

Let $\varphi(\mf{q})$ denote the cardinality of the unit group of $\mc{O}_K/\mf{q}$, i.e.,
\begin{equation}\label{phi}
	\varphi(\mf{a}) = N_{K}(\mf{a}) \prod_{\mf{p} | \mf{a}} \Big(1 - \frac{1}{N_{K}(\mf{p})} \Big) .
\end{equation}

Let $h = h((1))$ denote the class number of $K$. Let $U$ denote the unit group of $\mc{O}_K$ and $U_{\mathfrak{q}, 1}=\left\{\dc{a} \in \mathcal{O}_K^*: \dc{a} \equiv 1 \pmod{\mf{q}}, \dc{a} \succ 0\right\}$. Since $K$ has no real embeddings, and the units of $\mc{O}_K$ occupy distinct residue classes modulo $\mf{q}$ by assumption, the quantities $h(\mf{q})$ and $\varphi(\mf{q})$ are related by the following:
\begin{equation}\label{hphi}
	h(\mathfrak{q})= \varphi(\mathfrak{q}) \frac{h}{|U|},
\end{equation}
where $|U|=4$ if $K = \Q(i)$, $|U| = 3$ if $K = \Q(\sqrt{-3})$, and $|U| = 2$ otherwise.

For any character $\chi_0$ of $H^{\mf{q}}$, we define $\chi(\mf{a}) = \chi_0([\mf{a}])$ if $(\mf{a}, \mf{q})=1$ and $\chi(\mf{a})=0$ otherwise, and call $\chi$ a finite Hecke character modulo $\mf{q}$. Throughout this paper, $\chi$ will denote a finite order Hecke character of $K$. For $
\Re(s) > 1$, Hecke $L$-function $L(s,\chi)$ is defined by 
\[
L(s,\chi) = \sum_{\mf{a} \ideal \mc{O}_K} \frac{\chi(\mf{a})}{(N_K(\mf{a}))^s}.
\] 
If $\chi$ is nonprincipal, then $L(s,\chi)$ extends to an entire function, while if $\chi$ is principal, then $L(s,\chi)$ extends to a meromorphic function on the complex plane with a single simple pole at $s=1$.

\section{Page's Theorem and Bombieri--Vinogradov}\label{section-page-bv}
\begin{lemma}\label{landaupage}
(Landau-Page theorem for number fields). Let $Q \geqslant 100$. Suppose that $L(s, \chi)=0$ for some primitive character $\chi$ of modulus $\mf{q}$, $N_{K}(\mf{q}) \leq Q$, and some $s=\sigma+i t$. Then, we have
$$
1-\sigma \gg \frac{1}{\log (Q(1+|t|))},
$$
or else $t=0$ and $\chi$ is a quadratic character $\chi_Q$, which is unique.
\end{lemma}
\begin{proof}
This follows by combining \cite{lmo}*{Lemma 2.3} and \cite{hoffstein-ramakrishnan}*{Theorem A}.  
\end{proof}

\begin{corollary}\label{landaupagecorollary}
Let $Q \geqslant 100$. Then there exists an ideal $\mf{B}_Q$ which either is equal to $(1)$ or is a prime with the property that
$$
1-\sigma \gg \frac{1}{\log (Q(1+|t|))}
$$
whenever $L(\sigma+i t, \chi)=0$ and $\chi$ is a character mod $\mf{q}$ with $N_{K}(\mf{q}) \leq Q$ and $\mf{q}$ coprime to $\mf{B}_Q$.
\end{corollary}
\begin{proof}
This follows from Lemma \ref{landaupage} with $\mf{B}_Q$ the prime factor of largest norm of the conductor of $\chi_Q$. (If no such $\chi_Q$ exists, set $\mf{B}_Q = (1)$.)
\end{proof}

A linear form is a function $L: \mc{O}_K \to \mc{O}_K$ of the form $\dc{l}_1 \dc{z}  + \dc{l}_2$ with $\dc{l}_1, \dc{l}_2 \in \mc{O}_K$ and $\dc{l}_1\neq 0$. Define
$\psi(x, \chi)=\psi_0(x, \chi)=\sum_{N_{K}(\mathfrak{a}) \leq x} \Lambda(\mathfrak{a}) \chi(\mathfrak{a}), \quad \psi_k(x, \chi)=\int_1^x \psi_{k-1}(z, \chi) \frac{d z}{z} \quad($ for $k \geqslant 1)$.
Let
\begin{equation}
\psi_0(x, \mathfrak{a}, \mathfrak{q}) = \psi(x, \mathfrak{a}, \mathfrak{q})=\sum_{\substack{N_{K}(\mf{b}) \leq x \\ \mathfrak{b} \equiv \mathfrak{a} \pmod{\mathfrak{q}}}} \Lambda(\mathfrak{b}),
\end{equation}
and similarly, define
$$
\psi_k(x, \mathfrak{a}, \mathfrak{q})=\int_1^x \psi_{k-1}(z, \mathfrak{a}, \mathfrak{q}) \frac{d z}{z}.
$$
\begin{lemma}
For any ideal $\mf{q} \ideal \mc{O}_K$ and any $\dc{a} \in \mc{O}_K$,
	\begin{align*}
	\#\{\mf{p} \ideal \mc{O}_K: N_{K}(\mf{p}) \leq z, \mf{p} \equiv \dc{a} \pmod{\mf{q}} \} = \#\{\dc{p} \in \mc{O}_K: N_{K}(\dc{p}) \leq z, \dc{p} \equiv \dc{a} \pmod{\mf{q}} \}.
	\end{align*}
	Consequently, we can unambiguously define
	\begin{align*}
	\pi(z; \mf{q}, \dc{a}) &:= \#\{\mf{p} \ideal \mc{O}_K: N_{K}(\mf{p}) \leq z, \mf{p} \equiv \dc{a} \pmod{\mf{q}} \} \\
	&= \#\{\dc{p} \in \mc{O}_K: N_{K}(\dc{p}) \leq z, \dc{p} \equiv \dc{a} \pmod{\mf{q}} \}.
	\end{align*}
 
\end{lemma}
\begin{proof}
Recall that we assumed in Section 2 that for units $\dc{u}, \dc{u}' \in \mc{O}_K^{\times}$, we have $\dc{u} \not\equiv \dc{u}' \pmod{\mf{q}}$. It follows that for any ideal $\mf{p}$ with $\mf{p} \equiv \dc{a}\pmod{\mf{q}}$, there is a \textit{unique} generator $\dc{p}$ of $\mf{p}$ with $\dc{p} \equiv \dc{a} \pmod{\mf{q}}$; define $f(\mf{p}) = \dc{p}$. 

It is evident that the function $f$ is injective. Furthermore, for any $\dc{p} \in \mc{O}_K$ with $\dc{p} \equiv \dc{a} \pmod{\mf{q}}$, the ideal $\mf{p} = (\dc{p})$ is a principal ideal with $\mf{p} \equiv \dc{a} \pmod{\mf{q}}$. Thus, we have established the bijection below, which proves the lemma:
\begin{align*}
 \#\{\mf{p} \ideal \mc{O}_K: N_{K}(\mf{p}) \leq z, \mf{p} \equiv \dc{a} \pmod{\mf{q}} \} \xleftrightarrow{} \#\{\dc{p} \in \mc{O}_K: N_{K}(\dc{p}) \leq z, \dc{p} \equiv \dc{a} \pmod{\mf{q}} \}.
\end{align*}
\end{proof}
The main result of this section is the following: 
\begin{lemma}\label{page-bv}
Fix $\eps > 0$. Let $x$ be a large quantity. Let $Q = \exp(c_1 \sqrt{\log x})$. Then, there exists an ideal $\mf{B}$ of $\mc{O}_K$ satisfying $N_{K}(\mf{B}) \leq x$, which is either $(1)$ or a prime, such that	
\begin{equation}\label{assume}
\sideset{}{'}\sum_{\substack{N_{K}(\mf{q})<x^{1 / 3-\epsilon} \\
(\mf{q}, \mf{B})=1}} \sup _{\substack{(\dc{a}, \mf{q})=1 \\
z \leq x \log ^4 x}}\left|\pi(z ; \mf{q}, \dc{a})-\frac{\Li(z)}{h(\mf{q})}\right| = O_\eps \left( x \exp(-c\sqrt{\log x})\right).
\end{equation}
\end{lemma}
\begin{proof}
 Let $\mf{B}$ be the quantity $\mf{B}_Q$ guaranteed by Corollary \ref{landaupagecorollary} with this value of $Q$. For the remainder of this proof, the implied constants may depend on $\eps$. By the display following \cite{wilson}*{(51)}, we have that (if $T(\mf{q})$ denotes the number of residue classes of $\mf{q}$ containing a unit),
\[
\begin{aligned}
\sideset{}{'}\sum_{D<N_{K}(\mathfrak{q}) \leq Q} & \max _{z \leq x \log^4 x} \max _{\substack{\mathfrak{a} \pmod{\mathfrak{q}} \\
(\mathfrak{a}, \mathfrak{q})=1}} \frac{1}{T(\mathfrak{q})}\left|\psi_3(z, \mathfrak{a}, \mathfrak{q})-\frac{z}{h(\mathfrak{q})}\right| \\
& \ll x D^{-1} \log ^{11} x+x^{2/3} D Q \log ^{2 n+9 } x+\frac{x Q}{T^3} \log^5 x .
\end{aligned}
\]
Furthermore, by Corollary \ref{landaupagecorollary} (combined with a generalization of the explicit formula for $\psi(z,\chi)$ in \cite{davenport}*{Chapter 19} to finite order Hecke characters of imaginary quadratic fields, which can be proved in the same manner as for $\Q$; see \cite{gaussianhecke}{Section 2.9} for the explicit formula when $K=\Q(i)$, and see  \cite{lo}{Section 9} for a more general statement applying for arbitrary Hecke characters (not necessarily finite order) of any number field), we have that there exists some (small) $c$ such that whenever $1 < N_{K}(\mf{q}) \leq \exp(6c\sqrt{\log x})$, $z \leq x \log^4(x)$ and $(\mf{q},\mf{B})=1$, 
\begin{equation*}
\frac{1}{\varphi(\mf{q})} \sideset{}{^*}\sum_{\chi}|\psi(z,\chi)| \ll x \exp(-9c\sqrt{\log x}),	
\end{equation*}
where the asterisk over the sum above indicates that it is restricted to primitive Hecke characters of $H^{\mf{q}}$.

Choosing  $D = \exp(5c\sqrt{\log x})$, $T = x^{1/9}$ and $Q = x^{1/3-\eps}$, we find that (since $T(\mf{q}) = |U|$ is a constant depending on $K$, by our assumption that the units of $K$ occupy distinct residue classes modulo $\mf{q}$),
\[
\begin{aligned}
\sum_{D<N_{K}(\mathfrak{q}) \leq Q}^{\prime} & \max _{z \leq x \log^4 x} \max _{\substack{\mathfrak{a} \pmod{\mathfrak{q}} \\
(\mf{a}, \mathfrak{q})=1}} \left|\psi_3(z, \mathfrak{a}, \mathfrak{q})-\frac{z}{h(\mathfrak{q})}\right| \ll x\exp(-4c\sqrt{\log x}). 
\end{aligned}
\] 
Applying the same unsmoothing argument as in \cite{gallagher}{Page 6}, we find that
\[
\begin{aligned}
\sum_{D<N_{K}(\mathfrak{q}) \leq Q}^{\prime} & \max _{z \leq x \log^4 x} \max _{\substack{\mathfrak{a} \pmod{\mathfrak{q}} \\
(\mf{a}, \mathfrak{q})=1}} \left|\psi(z, \mathfrak{a}, \mathfrak{q})-\frac{z}{h(\mathfrak{q})}\right| \ll x\exp(-3c\sqrt{\log x}). 
\end{aligned}
\] 
It follows that
\[
\begin{aligned}
& \sideset{}{'}\sum_{\substack{N_{K}(\mf{q})<x^{1 / 3 -\epsilon} \\
(\mf{q}, \mf{B})=1}} \sup _{\substack{(\dc{a}, \mf{q})=1 \\
z \leq x \log ^4 x}}\left|\pi(z ; \mf{q}, \dc{a})-\frac{\Li(z)}{h(\mf{q})}\right| \ll x \exp (-c \sqrt{\log x})+\log x \\
& \times \sum_{\substack{N_{K}(\mf{q})\leq \exp (6 c \sqrt{\log x}) \\
(\mf{q}, \mf{B})=1}} \sum_\chi^* \sup _{z \leq x \log ^4 x} \frac{|\psi(z, \chi)|}{h(\mf{q})} \ll x\exp(-c\sqrt{\log x}).
\end{aligned}
\]
\end{proof}

\section{Rankin argument}\label{section-rankin}
Define $\mc{P}(x) = \prod_{N_{K}(\mf{p}) \leq x} \mf{p}$, and $P(x) = N_{K}(\mc{P}(x))$.
\begin{lemma}\label{y12}
	Let $x$ be a positive integer. Define $Y_1(x)$ to be the largest integer with the property that there exists a ball of radius $Y_1(x)$ such that all elements of the ball are divisible by a prime ideal of norm at most $x$.
	
	 Define $Y_2(x)$ to be the largest integer such that there exist residue classes $\dc{a}_{\mf{p}}$ for each prime ideal $\mf{p}$ of norm at most $x$ such that the set $\{\dc{z} \in \mc{O}_K: \dc{z} 
	 \equiv \dc{a}_{\mf{p}} \text{ for some } \mf{p} \text{ with } N_{K}(\mf{p}) \leq x\}$ contains a ball of radius $Y_2(x)$. Then $Y_1(x) = Y_2(x)$.
\end{lemma}
\begin{proof}
	First, we prove that $Y_1(x) \leq Y_2(x)$. Suppose that there exists a ball $B(\dc{x}_0,Y_1(x))$ such that all elements of the ball are divisible by a prime ideal of norm at most $x$. For each $\mf{p}$ with $N_K(\mf{p}) \leq x$, let $\dc{a}_{\mf{p}}$ be the congruence class of $-\dc{x}_0\pmod{\mf{p}}$. For any $\dc{z}$ with $N_{K}(\dc{z}) \leq Y_1(x)$, the element $\dc{x}_0 + \dc{z}$ of the ball $B(\dc{x}_0, Y_1(x))$ is divisible by $\mf{p}$ for some $\mf{p}$ with $N_K(\mf{p}) \leq x$, meaning that $\dc{x}_0 + \dc{z} \equiv 0 \pmod{\mf{p}}$, i.e., $\dc{z} \equiv -\dc{x}_0 \equiv \dc{a}_{\mf{p}} \pmod{\mf{p}}$. It follows that the set  $\{\dc{z} \in \mc{O}_K: \dc{z} 
	 \equiv \dc{a}_{\mf{p}} \text{ for some } \mf{p} \text{ with } N_{K}(\mf{p}) \leq x\}$ contains $B(0,Y_1(x))$.
	
	Second, we prove that $Y_2(x) \leq Y_1(x)$. Suppose that there exist residue classes $\dc{a}_{\mf{p}}$ for each prime ideal of norm at most $x$ such that the set $\{\dc{z} \in \mc{O}_K: \dc{z} 
	 \equiv \dc{a}_{\mf{p}} \text{ for some } \mf{p} \text{ with } N_{K}(\mf{p}) \leq x\}$ contains a ball $B(\dc{x}_0,Y_2(x))$ of radius $Y_2(x)$. Then, by the Chinese Remainder Theorem, there exists an element $\dc{y}_0$ that is congruent to $-\dc{a}_{\mf{p}}\pmod{\mf{p}}$ for each $\mf{p}$. If $\dc{z} \in \mc{O}_K$ with $N_K(\dc{z}) \leq Y_2(x)$, then for each $\mf{p}$ with $N_K(\mf{p}) \leq x$ we have that $\dc{z} + (\dc{y}_0+\dc{x}_0) = \dc{y}_0 + (\dc{x}_0+\dc{z}) \equiv -\dc{a}_{\mf{p}} + (\dc{x}_0 + \dc{z})~\pmod{\mf{p}}$. By assumption, for some $\mf{p}$ with $N_K(\mf{p}) \leq x$, we have that $\dc{x}_0 + \dc{z} \equiv \dc{a}_{\mf{p}} \pmod{\mf{p}}$. It follows that all elements of the ball $B(\dc{y}_0+\dc{x}_0,Y_2(x))$ are divisible by a prime ideal of norm at most $x$, and hence that $Y_2(x) \leq Y_1(x)$.
\end{proof}
Since $Y_1(x) = Y_2(x)$, we henceforth define $Y(x) = Y_1(x) = Y_2(x)$. The following lemma (cf. \cite{pollack}{pg. 4}) will be used throughout the paper:
\begin{lemma}\label{elements}
	Let $K$ be a quadratic field. The number of elements $\dc{u}$ of $\mc{O}_K$ satisfying a congruence condition $\dc{u} \equiv \dc{a}\pmod{\mathfrak{q}}$ and $N_K(\dc{u}) \leq x$ is 
	\begin{equation*}
	\frac{x}{N_{K}(\mathfrak{q})} + O\Big(1 + \Big( \frac{x}{N_{K}(\mathfrak{q})} \Big)^{1/2}\Big).
	\end{equation*}
\end{lemma}
We record the following consequence of the lemma above:
\begin{corollary}\label{nottoofar}
	Let $K$ be an imaginary quadratic field. Then, for any ideal $\mf{q}$ and any residue class $\dc{a} \pmod{\mf{q}}$, there exists a nonzero element of $\mc{O}_K$ in the residue class $\dc{a}\pmod{\mf{q}}$ with norm $O(N_{K}(\mf{q}))$.
\end{corollary}	
\begin{proof}
This follows immediately from the fact that the main term in Lemma \ref{elements} is larger than the error term when $x \gg N_K(\mf{q})$.
\end{proof}
Let 
\begin{small}
\begin{equation}\label{gdef}
	G(x) = \max\{y: \text{There exists } \dc{x}_0 \in \mc{O}_K \text{ with } N_{K}(\dc{x}_0) \leq x \text{ and } B(\dc{x}_0,y) \cap \{\dc{p} \in \mc{O}_K: \dc{p} \text{ prime}\} = \emptyset\}.
\end{equation}
\end{small}

By Lemma \ref{y12}, there exists some element $\dc{a}_0$ of $\mc{O}_K$ such that every element of $B(\dc{a}_0,Y(x))$ is divisible by a prime ideal of norm at most $x$. By Corollary \ref{nottoofar}, there exists an element $\dc{a}_1 \neq 0$ of $\mc{O}_K$ of norm $O(P(x))$ with this property. Similarly, by Corollary \ref{nottoofar}, there also exists a nonzero element $\dc{b}$ in the ideal $\mc{P}(x)$, which necessarily has norm at least $P(x)$ and at most $O(P(x))$. By the triangle inequality, there exists some positive integer $n = O(1)$ such that $N_{K}(n\dc{b}+\dc{a}_1)$ is bounded below by $10P(x)$ and above by $O(P(x))$. Set $\dc{a} = n \dc{b} + \dc{a}_1$. Since we trivially have that $Y(x) \leq P(x)$, it follows that every element of the ball $B(\dc{a},Y(x))$ is of norm at least $P(x)$. Since any element of this ball is divisible by a prime ideal of norm at most $x$, it follows that any element of this ball is composite. In particular, it follows that $G(N_K(\dc{a})) \geq Y(x)$. By Theorem \ref{landauprimeideal}, $\log P(x) = (1+o(1))x$. Setting $y = N_{K}(\dc{a})$, we obtain that

\begin{equation}
	G(y) \geq Y((1+o(1)\log(y)).
\end{equation}

To prove Theorem \ref{mainthm}, it therefore suffices to show that 
\begin{equation}\label{suffices}
Y(x) \gg x \frac{\log x}{\log_2 x} \log_3 x. 
\end{equation}

We require the following result regarding smooth ideals in number fields, which is \cite{thorne}{Lemma 5.4}.
\begin{lemma}\label{lemma5.4}
Let $\Psi_K(x, y)$ be the number of ideals of norm $<x$ which are composed only of primes with norm $<y$, and write $u:=\log x / \log y$. Then for $1 \leq u \leq$ $\exp \left(c(\log y)^{3 / 5-\epsilon}\right)$ (for a certain constant $\left.c\right)$ we have
$$
\Psi_K(x, y) \ll x \log ^2 y \exp (-u(\log u+\log \log u+O(1)))
$$
\end{lemma}

Let
\begin{equation}\label{ydef}
y:=\left\lfloor c x \frac{\log x}{\log _2 x} \log_3x\right\rfloor,
\end{equation}
and
\[
z_0:=x^{\log _3 x /\left(5 \log _2 x\right)}.
\]
We will show that $Y(x) \gg y - x$ by covering the set $\{\dc{z} \in \mc{O}_K: x < N_{K}(\dc{z}) \leq y\}$ with residue classes modulo prime ideals of norm at most $x$. By \eqref{suffices}, this suffices to prove Theorem \ref{mainthm}. To this end, we introduce one set of prime ideals of $K$, and two sets of prime elements:
\[
\begin{aligned}
& \mathcal{S}:=\left\{\mf{s} \ideal \mc{O}_K \text { prime }: \log ^{20} x< N_{K}(\mf{s}) \leq z_0\right\} \\
& \mathcal{P}:=\{\dc{p} \in \mc{O}_K \text { prime }: x / 2< N_{K}(\dc{p}) \leq x\} \\
& \mathcal{Q}:=\{\dc{q} \in \mc{O}_K \text { prime }: x< N_{K}(\dc{q}) \leq y\} .
\end{aligned}
\]
Correspondingly, we define the following sifted sets of elements of $\mc{O}_K$.
\[
\begin{aligned}
& S(\vec{\dc{a}}):=\left\{\dc{n} \in \mc{O}_K: \dc{n} \not \equiv \dc{a}_s \pmod{\mf{s}}\right.\text{ for all }\left.s \in \mathcal{S}\right\} \\
& T(\vec{b}):=\left\{\dc{n} \in \mc{O}_K: \dc{n} \not \equiv \dc{b}_{\dc{p}}\pmod {\dc{p}}\right.\text{ for all }\left.\dc{p} \in \mathcal{P}\right\}.
\end{aligned}
\]
We reduce the main theorem to the following.
\begin{theorem}\label{sieved}
There exist vectors $\vec{\dc{a}} = (\dc{a}_s \pmod{\mf{s}})_{\mf{s} \in \mc{S}}$ and $\vec{b} = (\dc{b}_{\dc{p}} \pmod{\dc{p}})_{\dc{p} \in \mc{P}}$ such that
\begin{equation*}
|Q \cap S(\vec{\dc{a}}) \cap T(\vec{b})| \leq \frac{x}{5 \log x}.
\end{equation*}
\end{theorem}
\begin{proof}[Proof of Theorem \ref{mainthm} assuming Theorem \ref{sieved}]
We first set 
\[
\dc{a}_{\mf{p}}=0 \quad\left(N_{K}(\mf{p}) \leq \log ^{20} x, z_0 <N_{K}(\mf{p}) \leq x / 4\right) .
\]
We let $\vec{\dc{a}}$ and $\vec{b}$ be as in Theorem \ref{sieved}. Let
\[
V=\left\{\dc{n} \in \mc{O}_K: \dc{n} \not \equiv 0(\pmod{\mf{p}}) \text { for all } \mf{p}, N_{K}(\mf{p}) \leq \log ^{20} x \text { and } z_0 < N_{K}(\mf{p}) \leq x / 4\right\}
\]
Consider the set
\[
\mathcal{U}:= \{\dc{z} \in \mc{O}_K: N_{K}(\dc{z}) \in (x, y]\} \cap S(\vec{\dc{a}}) \cap T(\vec{b}) \cap V.
\]
Any element of $\mc{U}$ is either composed of prime ideals of norm at most $z_0$ (i.e., is ``$z_0$-smooth), or is divisible by a prime ideal of norm larger than $x/4$. Since all prime factors of elements of $\mc{U}$ have norm larger than $\log^{20}(x)$ and since all elements of $\mc{U}$ have norm $O(x\log x)$, it follows that the elements of $\mc{U}$ are either $z$-smooth or prime. In other words, $\mc{U}$ differs from $Q \cap S(\vec{\dc{a}}) \cap T(\vec{b})$ by a set of $z_0$-smooth numbers. However, Lemma \ref{lemma5.4} implies that the set of $z_0$-smooth numbers in $\mc{O}_K$ with norm at most $y$ is $O(x/\log^2(x))$. Consequently, we find that
\begin{equation*}
|\mc{U}| \leq (1+o(1))\frac{x}{5\log x}.	
\end{equation*}
We cover the remaining residue classes by matching them to the primes $\mf{p}$ with $x/4 < N_{K}(\mf{p}) \leq x/2$. (There are enough such prime ideals by Theorem \ref{landauprimeideal}.)
\end{proof}

\section{Sieve weights}\label{section-sieve-weights}
In this section, we develop a number field variant of the arguments in \cite{denseclusters}, which yields Proposition \ref{prop7.14}. We then use this to create good sieve weights analogous to \cite{fgkmt} in Theorem \ref{thm8.6}.

Let $N \geq 100$, and $\theta < 1$. Assume that the parameters $s,R,D,z,k$ satisfy

\begin{align*}
&\frac{\log _2 N}{2} \leq s \leq 2 \log _2 N, \quad N^{\frac{\theta}{4}-\frac{2}{s}} \leq R \leq N^{\frac{\theta}{4}-\frac{1}{s}}, \quad D=R^{1 / s}, \\
& 2N \leq \tilde{N} \leq N (\log N)^3, \quad  (\log N)^{9999k^2} \leq z \leq (\log N)^{99999k^2}.
\end{align*}
Assume further that $z$ is larger than any prime dividing $\text{disc}(K)$.
Define
\begin{small}
\begin{equation}\label{mcddef}
	\mathscr{D}:=\left\{\vec{\mf{d}} \in \{\mf{a} \ideal \mc{O}_K\}^k: \rad(N_{K}(\mf{d}_1 \cdots \mf{d}_k)) \leq R, \mu^2\left(\mf{d}_1 \cdots \mf{d}_k\right)=1, P^{-}(N_K\left(\mf{d}_1 \cdots \mf{d}_k\right))>z\right\}.
\end{equation}
\end{small}
Let $\mf{B}$ be the ideal (which is either $(1)$ or prime) guaranteed by Lemma \ref{page-bv} with $x = \tilde{N}$. We require the fundamental lemma of sieve theory, which we will use in the following form:
\begin{theorem}{\cite{opera}{Proposition 6.7}}\label{thm3.6}
For any pair $(z,D)$ of positive integers with $2 \leq z \leq D^{1/2}$, there are sieves $\lambda^+$ and $\lambda^-$ satisfying
\begin{equation}\label{lambdaplus}
\lambda_1^{+}=1, \quad \sum_{d \mid m} \lambda_d^{+} \geqslant 0 \quad(m>1)	
\end{equation}
and
\begin{equation}\label{lambdaminus}
\lambda_1^{-}=1, \quad \sum_{d \mid m} \lambda_d^{+} \leq 0 \quad(m>1)	,
\end{equation}
satisfying $|\lambda^{\pm}| \leq 1$ for all $d$, with support in $\ms{D}(z,D) := \{d \in \N: \mu^2(d) = 1, P^+(d) \leq z, d \leq D\}$. Furthermore, for any multiplicative function $g$, if there exist constants $\kappa \geq 0$ and $B > 0$ such that
\begin{equation}\label{Omega}
	\prod_{y \leq p \leq w}(1-g(p))^{-1} \leq\left(\frac{\log w}{\log y}\right)^\kappa \exp \left(\frac{B}{\log y}\right) \quad(2 \leq y \leq w \leq z),
\end{equation}
then, with $s = \max(100, \frac{\log D}{\log z})$, we have that
\begin{equation}\label{flparta}
\sum_d \lambda_d^{ \pm} g(d)=\left(1+O\left(e^{-s \log s+s \log _3 s+O_{\kappa, B}(s)}\right)\right) \prod_{p \leq z}(1-g(p)).	
\end{equation}
\end{theorem}
Let $(\dc{\dc{a}_1n + \dc{b}_1}, \dots \dc{a}_k n + \dc{b}_k)$ be a tuple of linear forms. Let
\begin{equation}\label{edef}
\begin{split}
\dc{E}&=\dc{E}(\vec{\dc{a}}, \vec{\dc{b}})=\prod_{i=1}^k \dc{a}_i \prod_{i<j}\left(\dc{a}_i \dc{b}_j-\dc{a}_j \dc{b}_i\right), \\
\mathscr{E}&=\mathscr{E}(\vec{\dc{a}}, \vec{\dc{b}})=\left\{\vec{\mf{d}} \in \{\mf{a} \ideal \mc{O}_K\}^k:\left(\mf{d}_1 \cdots \mf{d}_k, \dc{E} \mf{B}\right)=1\right\}.
\end{split}
\end{equation}
Let
\begin{equation}\label{rhodef}
	\rho(\mf{d})=\#\left\{\dc{n} \pmod{\mf{d}}:\left(\dc{\dc{a}_1 n+\dc{b}_1}\right) \cdots\left(\dc{a}_k \dc{n}+\dc{b}_k\right) \equiv 0 \pmod{\mf{d}}\right\}.
\end{equation}
When $\rho(\mf{d}) < N_{K}(\mf{d})$, we say that the collection $(\dc{a}_i \dc{n} + \dc{b}_i)_{i=1}^k$ is \textit{admissible}. We assume that $(\dc{a}_i \dc{n} + \dc{b}_i)_{i=1}^k$ is indeed admissible. Define \[ H = N_K(\mf{B}) \cdot \disc(K).\]
Let
\begin{equation}\label{vdef}
	V=\prod_{\substack{\rad(N_{K}(\mf{p})) \leq z \\ \mf{p} \nmid H}}\left(1-\frac{\rho(\mf{p})}{N_{K}(\mf{p})}\right).
\end{equation}
Let $\mu^+$ denote an upper bound sieve satisfying \eqref{lambdaplus} with respect to the parameters $z, D$ (in particular, we have $|\mu^+(n)| \leq 1$ for all $n \in \Z $). Let $\lambda: \{\mf{a} \ideal \mc{O}_K\}^k \to \R$ be a function supported on $\ms{D}$ satisfying 
\begin{equation}\label{eq7.13-gaps}
|\lambda(\vec{\mf{d}})| \leq 1.
\end{equation}
 Then, define
\begin{equation}\label{wdef}
\begin{split}
w(\dc{n})&=\sum_{\substack{\mf{t} \mid \left(\dc{a}_1 \dc{n}+\dc{b}_1\right) \cdots\left(\dc{a}_k \dc{n}+\dc{b}_k\right) \\ P^+(N_K(\mf{t})) \leq z, (\mf{t}, H) = 1}} \mu^{+}(\rad(N_{K}(\mf{t})))\cdot \frac{\mu(\mf{t})}{\mu(\rad(N_{K}(\mf{t})))} \\
&\qquad \cdot \Bigg(\sum_{\substack{\vec{\mf{d}} \in \mathscr{D} \cap \mathscr{E} \\ \forall: \mf{d}_j \mid \dc{a}_j \dc{n}+\dc{b}_j}} \lambda(\vec{\mf{d}})\Bigg)^2	.
\end{split}
\end{equation}
\begin{lemma}\label{rp}
Suppose that $\mf{m}_j \mid \dc{a}_j \dc{n}+\dc{b}_j$ for every $j$ and $\mathbf{m} \in \mathscr{E}$. Then $\mf{m}_1, \ldots, \mf{m}_k$ are pairwise relatively prime, and $\left(\mf{m}_j, \dc{a}_j\right)=1$ for all $j$.
\end{lemma}
\begin{proof}
If $\mf{p} \mid \mf{m}_i$ and $\mf{p} \mid \mf{m}_j$ then $\mf{p}$ divides $\dc{a}_i\left(\dc{a}_j \dc{n}+\dc{b}_j\right)-\dc{a}_j\left(\dc{a}_i \dc{n}+\dc{b}_i\right)=\dc{a}_i \dc{b}_j-\dc{a}_j \dc{b}_i$, and so $\mf{p} \mid E$. This proves the first claim. Since $\left(\dc{a}_1 \dc{n}+\dc{b}_1, \ldots, \dc{a}_k n+\dc{b}_k\right)$ is an admissible set, $\left(\dc{a}_j, \dc{b}_j\right)=1$ for all $j$. Hence, if $\mf{p}$ is prime, $\mf{p}\left|\mf{m}_j\right|\left(\dc{a}_j \dc{n}+\dc{b}_j\right)$ and $\mf{p} \mid \dc{a}_j$ then $\mf{p} \mid \dc{b}_j$, a contradiction. Thus, $\left(\dc{a}_j, \mf{m}_j\right)=1$.
\end{proof}

\begin{proposition}\label{prop7.8}
Let $\mu^{+}$be an upper bound sieve function from Theorem \ref{thm3.6} with parameters $z, D$. Let $\lambda(\vec{\mf{d}})$ satisfy $|\lambda(\vec{\mf{d}})| \leq 1$ and be supported on $\mathscr{D}$. For $\vec{\mf{r}} \in \mathscr{D}$ define
\begin{equation}\label{xidef}
\xi(\vec{\mf{r}})=\sum_{\vec{\mf{d}}} \frac{\lambda\left(\mf{r}_1 \mf{d}_1, \ldots, \mf{r}_k \mf{d}_k\right)}{N_{K}(\mf{d}_1) \cdots N_{K}(\mf{d}_k)} .
\end{equation}

Let $\left(\dc{a}_1 \dc{n}+\dc{b}_1, \ldots, \dc{a}_k \dc{n}+\dc{b}_k\right)$ be an admissible set of linear forms, with $k \leq (\log N)^{1/9}$ and $k$ larger than a suitable (absolute) constant, and such that
\begin{equation}\label{absizes}
1 \leq N_{K}(\dc{a}_i)\leq N^2, \quad N_{K}(\dc{b}_i) \leq N^2 \quad(1 \leq i \leq k) .
\end{equation}

Define $\dc{E}, \mathscr{E}$ by \eqref{edef}, $V$ by \eqref{vdef} and $w(\dc{n})$ by \eqref{wdef}. Then
$$
\sum_{N<N_{K}(\dc{n}) \leq 2 N} w(\dc{n})=V N \sum_{\vec{\mf{r}} \in \mathscr{D}} \frac{\xi(\vec{\mf{r}})^2}{N_{K}(\mf{r}_1) \cdots N_{K}(\mf{r}_k)}+O\left(\frac{N}{(\log N)^{9990 k^2}}\right)
$$
\end{proposition}
\begin{proof}
By expanding the square in the definition of $w(\dc{n})$ and interchanging the order of summation, we have that
\begin{align*}
\sum_{N<N_{K}(\dc{n}) \leq 2 N} w(\dc{n}) &=\sum_{\substack{\rad(N_{K}(\mf{t})) \leq D \\ P^+(N_K(\mf{t}) )\leq z, (\mf{t}, H) = 1}} \mu^{+}(\rad(N_{K}(\mf{t})) )  \cdot \frac{\mu(\mf{t})}{\mu(\rad(N_{K}(\mf{t})))} \\
&\qquad \cdot \sum_{\vec{\mf{d}}, \vec{\mf{e}} \in \mathscr{D} \cap \mathscr{E}} \lambda(\vec{\mf{d}}) \lambda(\vec{\mf{e}}) \sum_{\substack{N<N_{K}(\dc{n}) \leq 2 N \\ \mf{t} \mid\left(\dc{a}_1 n+\dc{b}_1\right) \cdots\left(\dc{a}_k n+\dc{b}_k\right) \\ \forall j:\left[\mf{d}_j, \mf{e}_j\right]\left(\dc{a}_j n+\dc{b}_j\right)}} 1.
\end{align*}
Since $\vec{\mf{d}},\vec{\mf{e}} \in \ms{E}$, Lemma \ref{rp} implies that $(\mf{d}_i\mf{e}_i, \mf{d}_j\mf{e}_j)= 1$ for $i \neq j$ and $(\mf{d}_i\mf{e}_i, \dc{a}_i) = 1$ for all $i$. Consequently, the conditions $[\mf{d}_j, \mf{e}_j] \mid \dc{a}_j n + \dc{b}_j$ define a single residue class mod $\prod_i [\mf{d}_i, \mf{e}_i]$. By definition, the condition $ \mf{t} \mid\left(\dc{a}_1 n+\dc{b}_1\right) \cdots\left(\dc{a}_k n+\dc{b}_k\right)$ defines $\rho(\mf{t})$ residue classes modulo $\mf{t}$. 

Furthermore, $P^+(N_K(\mf{t})) \leq z < P^-(N_K(\mf{d}_j\mf{e}_j))$, so the conditions $[\mf{d}_j, \mf{e}_j] \mid \dc{a}_j n + \dc{b}_j$ define a single residue class mod $\prod_i [\mf{d}_i, \mf{e}_i]$ and  $ \mf{t} \mid\left(\dc{a}_1 n+\dc{b}_1\right) \cdots\left(\dc{a}_k n+\dc{b}_k\right)$ defines $\rho(\mf{t})$ residue classes modulo $\mf{t} \prod_i [\mf{d}_i, \mf{e}_i]$. Thus, by Lemma \ref{elements}, we have that
\begin{small}
\begin{align*}
	&\sum_{N<N_{K}(\dc{n}) \leq 2 N} w(\dc{n}) \\
	& =\sum_{\substack{\rad(N_{K}(\mf{t})) \leq D \\ P^+(N_K(\mf{t})) \leq z,  (\mf{t}, H) = 1}} \mu^{+}(\rad(N_{K}(\mf{t}))) \cdot \frac{\mu(\mf{t})}{\mu(\rad(N_{K}(\mf{t})))}  \\ & \cdot \sum_{\vec{\mf{d}}, \vec{\mf{e}} \in \mathscr{D} \cap \mathscr{E}} \lambda(\vec{\mf{d}}) \lambda(\vec{\mf{e}})\left(\frac{\rho(\mf{t}) N}{N_{K}(\mf{t}\left[\mf{d}_1, \mf{e}_1\right] \cdots\left[\mf{d}_k, \mf{e}_k\right])}+O\left(\rho(\mf{t})\left(1 + \sqrt{\frac{N}{N_{K}(\mf{t}\left[\mf{d}_1, \mf{e}_1\right]\cdots\left[\mf{d}_k, \mf{e}_k\right])}}\right)\right)\right) \\ & =N V^{+} B+T.
\end{align*}
\end{small}
where
\begin{small}
\begin{align*}
V^{+} &=\sum_{\rad(N_{K}(\mf{t})) \leq D}\mu^{+}(\rad(N_{K}(\mf{t})))\cdot  \frac{\rho(\mf{t})}{N_{K}(\mf{t})} \cdot \frac{\mu(\mf{t})}{\mu(\rad(N_{K}(\mf{t})))} \cdot 1_{(\mf{t}, H) = 1}, \\
B &=\sum_{\vec{\mf{d}}, \vec{\mf{e}} \in \mathscr{D} \cap \mathscr{E}} \frac{\lambda(\vec{\mf{d}}) \lambda(\vec{\mf{e}})}{N_K(\left[\mf{d}_1, \mf{e}_1\right] \cdots\left[\mf{d}_k, \mf{e}_k\right])},
\end{align*}
\end{small}
and by \eqref{eq7.13-gaps} and the bound $|\mu^+(n)| \leq 1$,
$$
|T| \ll|\mathscr{D}|^2 \sum_{N_{K}(\mf{t}) \leq D^2} \rho(\mf{t}) + \sqrt{N} \sum_{N_{K}(\mf{t}) \leq D^2} \frac{\rho(\mf{t})}{\sqrt{N_{K}(\mf{t})}} \sum_{\vec{\mf{d}}, \vec{\mf{e}} \in \mathscr{D}} (N_K(\left[\mf{d}_1, \mf{e}_1\right]\cdots\left[\mf{d}_k, \mf{e}_k\right]))^{-1/2}.
$$

We will first estimate the error term $T$. For an ideal $\mf{a}$, let $\omega(\mf{a}) = \{\mf{p} \ideal \mc{O}_K : \mf{p} \mid \mf{a}\}$. Observe that 
$$
\begin{aligned}
|\mathscr{D}| & \leq \sum_{\substack{N_{K}(\mf{r}) \leq R^2 \\
P^{-}(N_K(\mf{r}))>z}} k^{\omega(\mf{r})} \mu^2(\mf{r}) \leq R^2 \sum_{\substack{N_{K}(\mf{r}) \leq R^2 \\
P^{-}(N_K(\mf{r}))>z}} \frac{k^{\omega(\mf{r})} \mu^2(\mf{r})}{N_{K}(\mf{r})} \\
& \leq R^2 \prod_{z<N_{K}(\mf{p}) \leq R^2}\left(1+\frac{k}{N_{K}(\mf{p})}\right) \leq R^2 \prod_{z<N_{K}(\mf{p}) \leq R^2}\left(1+\frac{1}{N_{K}(\mf{p})}\right)^k  \\
&\ll e^{O(k)} \cdot R^2 \cdot \left(\frac{\log R^2}{\log z} \right)^k \ll R^2 \cdot (\log N)^{k+1}.
\end{aligned}
$$

Furthermore, we have that
\begin{align*}
\sum_{N_{K}(\mf{t}) \leq D^2} \rho(\mf{t}) &\leq D^2 \sum_{N_{K}(\mf{t}) \leq D^2} \frac{\rho(\mf{t})}{N_{K}(\mf{t})} \leq D^2 \prod_{N_{K}(\mf{p}) \leq D^2}\left(1+\frac{\min\{k, N_{K}(\mf{p})-1\}}{N_{K}(\mf{p})}\right) \\
&= D^2 \prod_{N_{K}(\mf{p}) \leq 2k}\left(1+\frac{N_{K}(\mf{p})-1}{N_{K}(\mf{p})}\right) \prod_{2k < N_{K}(\mf{p}) \leq D^2}\left(1+\frac{k}{N_{K}(\mf{p})}\right)  \\
&\leq D^2 \cdot 4^{k} \cdot \exp(\sum_{2k < N_{K}(\mf{p}) \leq D^2} \log(1 + \frac{k}{N_{K}(\mf{p})}) ) \\
&\leq D^2 \cdot  \exp(O(k) + k \log\log \frac{D^2}{2k}  ) \ll  D^2 (\log N)^{k+1},
\end{align*}
and similarly,

\begin{align*}
\sum_{N_{K}(\mf{t}) \leq D^2} \frac{\rho(\mf{t})}{\sqrt{N_{K}(\mf{t})}} \leq D	 \sum_{N_{K}(\mf{t}) \leq D^2} \frac{\rho(\mf{t})}{N_{K}(\mf{t})}  \ll D (\log N)^{k+1}.
\end{align*}

Moreover,
\begin{align*}
	 \sum_{\vec{\mf{d}}, \vec{\mf{e}} \in \mathscr{D}} (N_K(\left[\mf{d}_1, \mf{e}_1\right]\cdots\left[\mf{d}_k, \mf{e}_k\right]))^{-1/2} &\leq R 	 \sum_{\vec{\mf{d}}, \vec{\mf{e}} \in \mathscr{D}} (N_K(\left[\mf{d}_1, \mf{e}_1\right]\cdots\left[\mf{d}_k, \mf{e}_k\right]))^{-1} \\
	 &\leq R \prod_{z < N_{K}(\mf{p}) \leq R^2} \left(1 + \frac{3}{N_{K}(\mf{p})} \right)^k \\
	 &\ll R \theta^{3k}	 \left(\frac{\log N}{9999k^2 \log_2(N)} \right)^{3k} \ll R (\log N)^{3k+1}.
\end{align*}

Hence,
$$
T \ll (R^4 D^2 + R D\sqrt{N}) (\log N)^{4k+3}  \ll N^\theta + N^{1 - \frac{1}{2s}} \ll \frac{N}{(\log N)^{9999 k^2}}.
$$

Next, note that
\begin{align*}
&\sum_{\substack{\rad(N_{K}(\mf{t})) \leq D \\ P^+(N_K(\mf{t})) \leq z, (\mf{t}, H) = 1}}\mu^{+}(\rad(N_{K}(\mf{t})))\cdot  \frac{\rho(\mf{t})}{N_{K}(\mf{t})} \cdot \frac{\mu(\mf{t})}{\mu(\rad(N_{K}(\mf{t})))} \\
& = \sum_{x \leq D} \mu^+(x) \frac{1}{\mu(x)}\sum_{\substack{\rad(N_{K}(\mf{t}))=x \\ (\mf{t}, H) = 1}}\frac{\rho(\mf{t})}{N_{K}(\mf{t})} \mu(\mf{t}).
\end{align*}

Define
\begin{align*}
g(x) = \frac{1_{(x,H)=1}}{\mu(x)} \sum_{\rad(N_{K}(\mf{t})) = x} \frac{\rho(\mf{t})}{N_{K}(\mf{t})} \mu(\mf{t}).	
\end{align*}

First, observe that since $\rho$, $N_K$ and $\mu$ are multiplicative functions on the set of ideals of $\mc{O}_K$,  $g$ is a multiplicative function on squarefree integers. Next, note that for any prime $p \in \Z$, there are either one or two prime ideals $\mf{p} \ideal \mc{O}_K$ lying above $p$. When there are two prime ideals $\mf{p}_1$ and $\mf{p}_2$ lying above $p$, there are three squarefree ideals $\mf{a} \ideal \mc{O}_K$ with $\rad(N_{K}(\mf{a})) = p$: the three ideals are $\mf{p}_1$, $\mf{p}_2$ and $\mf{p}_1\mf{p}_2$. On the other hand, when there is a single prime ideal $\mf{p}$ lying above $p$, then there is only one ideal $\mf{p}$ with $\rad(N_{K}(\mf{p})) = p$. In the first case, we have
\begin{align*}
	g(p) &= 1_{\substack{p \leq z \\ p \nmid H}} \left(\frac{\rho(\mf{p}_1)}{N_{K}(\mf{p}_1)} + \frac{\rho(\mf{p}_2)}{N_{K}(\mf{p}_2)} - \frac{\rho(\mf{p}_1 \mf{p}_2)}{N_{K}(\mf{p}_1 \mf{p}_2)}  \right)  \\
	1-g(p) &= \prod_{\substack{\rad(N_{K}(\mf{p}))=p \leq z \\ p \nmid H}} \left(1-\frac{\rho(\mf{p})}{N_{K}(\mf{p})} \right) .
\end{align*}

In the second case, we have that
\begin{align*}
	g(p) &= 1_{\substack{p \leq z \\ p \nmid H}} \left(\frac{\rho(\mf{p})}{N_{K}(\mf{p})} \right) \\ 
	1 - g(p) &= \prod_{\substack{\rad(N_{K}(\mf{p})) = p \leq z \\ p \nmid H}} \left(1 - \frac{\rho(\mf{p})}{N_{K}(\mf{p})} \right).
\end{align*}

Note that $g$ satisfies $(\Omega)$ with $\kappa=3k$ since $\rho(\mf{p}) \leq k$ for all $\mf{p}$ and for all primes $p \in \Z$, and since for any prime $p \in \Z$, the maximum possible number of squarefree ideals $\mf{a} \ideal \mc{O}_K$ with $\rad(N_{K}(\mf{a})) = p$ is $3$. 
Therefore, by the Fundamental Lemma (Theorem \ref{thm3.6}),
\begin{align*}
V^{+}&= V \cdot \left(1+O\left(e^{-\frac{1}{2} s \log s}\right)\right) = V+O\left(\frac{1}{(\log N)^{9999 k^2}}\right) .
\end{align*}

This is a genuine asymptotic since $V \gg e^{-O(k)} (\log z)^{-k} \gg e^{-O(k)} (99999k^2 \log_2(N))^{-k} \gg e^{-O(k(\log(k) + \log_3(N)))}$. We now turn to proving a preliminary upper bound for $B$. For any $\mf{m}_i=\left[\mf{d}_i, \mf{e}_i\right]$, there are at most $3^{\omega\left(\mf{m}_i\right)}$ choices for $\mf{d}_i, \mf{e}_i$. Hence, from \eqref{eq7.13-gaps},

\begin{equation}\label{eq7.19}
\left|\sum_{\vec{\mf{d}}, \vec{\mf{e}} \in \mathscr{D}} \frac{\lambda(\vec{\mf{d}}) \lambda(\vec{\mf{e}})}{N_{K}(\mf{m}_1) \cdots N_{K}(\mf{m}_k)}\right| \leq \prod_{i=1}^k \sum_{\substack{N_{K}(\mf{m}_i) \leq R^4 \\ P^{-}(N_K(\mf{m}_i)>z}} \frac{3^{\omega\left(\mf{m}_i\right)} \mu^2\left(\mf{m}_i\right)}{N_{K}(\mf{m}_i)} \ll (\log N)^{3k+1}.
\end{equation}

To asymptotically bound $B$, we first remove the conditions $\mf{d}, \vec{\mf{e}} \in \mathscr{E}$. Now from \eqref{edef} and \eqref{absizes},
$$
N_{K}(\dc{E}) \ll N^{2 k+4\left(k^2 / 2\right)}.
$$

Hence there are $\ll k^2 \frac{\log N}{\log z} \ll (\log N)/(\log_2 N)$ prime factors of $\dc{E}\mf{B}$ of norm larger than $z$. If $\vec{\mf{d}} \notin \mathscr{E}$ or $\vec{\mf{e}} \notin \mathscr{E}$ then there is a $\mf{p} \mid \dc{E} \mf{B}$ with $\mf{p} \mid \mf{m}_j$ for some $j$. Write $\mf{m}_j=\mf{p} \mf{m}_j^{\prime}$, then analogously to \eqref{eq7.19} we have
\begin{align*}
&\left|\sum_{\substack{\vec{\mf{d}}, \vec{\mf{e}} \in \mathscr{D} \\ \vec{\mf{d}} \notin \mathscr{E} \text { or } \vec{\mf{e}} \not\in \mathscr{E}}} \frac{\lambda(\vec{\mf{d}}) \lambda(\vec{\mf{e}})}{N_{K}(\mf{m}_1) \cdots N_{K}(\mf{m}_k)}\right| \\
&\leq \sum_{j=1}^k \sum_{\substack{\mf{p} \mid \dc{E}\mf{B}	 \\ N_{K}(\mf{p})>z}} \sum_{\substack{N_{K}(\mf{m}_j^{\prime}) \leq R^4 \\ P^{-}(N_K(\mf{m}_j^{\prime}))>z}} \frac{3^{\omega\left(\mf{m}_j^{\prime}\right)+1}}{N_{K}(\mf{m}_j^{\prime}) N_{K}(\mf{p})} \prod_{i \neq j} \sum_{\substack{N_{K}(\mf{m}_i) \leq R^4 \\ P^{-}(N_K(\mf{m}_i))>z}} \frac{3^{\omega\left(\mf{m}_i\right)}}{N_{K}(\mf{m}_i)} \\
&\ll k \cdot \left(O\left(\frac{\log R^2}{\log z}\right)\right)^{3k} \cdot \frac{1}{z} \cdot \frac{\log N}{\log_2 N} \ll \frac{(\log N)^{3k+2}}{(\log N)^{9999k^2}} \ll \frac{1}{(\log N)^{9998k^2}}.
\end{align*}
Therefore, we find that
$$
B=O\left(\frac{1}{(\log N)^{9998k^2}}\right)+B^{\prime}, \quad B^{\prime}=\sum_{\vec{\mf{d}}, \vec{\mf{e}} \in \mathscr{D}} \frac{\lambda(\vec{\mf{d}}) \lambda(\vec{\mf{e}})}{N_{K}(\left[\mf{d}_1, \mf{e}_1\right] \cdots\left[\mf{d}_k, \mf{e}_k\right])},
$$

and consequently,
\begin{equation}\label{eq7.21}
\sum_{N<N_{K}(\dc{n}) \leq 2 N} w(\dc{n})  =N V B^{\prime}+O\left(\frac{N}{(\log N)^{9998 k^2}}\right) .
\end{equation}

Finally, we estimate $B^{\prime}$. We begin with the following identity:
\begin{equation}\label{eq7.22}
\frac{1}{N_{K}([\mf{d}, \mf{e}])}=\frac{N_{K}((\mf{d}, \mf{e}))}{N_{K}(\mf{d} \mf{e})}=\frac{1}{N_{K}(\mf{d} \mf{e})} \sum_{\mf{r} \mid(\mf{d}, \mf{e})} \varphi(\mf{r}).
\end{equation}
From the above identity and the definition of $B'$, we obtain 
\[
\begin{aligned}
B^{\prime} & =\sum_{\vec{\mf{r}} \in \mathscr{D}} \varphi\left(\mf{r}_1\right) \cdots \varphi\left(\mf{r}_k\right)\left(\sum_{\forall j: \mf{r}_j \mid \mf{d}_j} \frac{\lambda(\vec{\mf{d}})}{N_{K}(\mf{d}_1 \cdots \mf{d}_k)}\right)\left(\sum_{\forall j: \mf{r}_j \mid \mf{e}_j} \frac{\lambda(\vec{\mf{e}})}{N_{K}(\mf{e}_1 \cdots \mf{e}_k)}\right) \\
& =\sum_{\vec{\mf{r}} \in \mathscr{D}} \frac{\varphi\left(\mf{r}_1\right) \cdots \varphi\left(\mf{r}_k\right)}{N_{K}(\mf{r}_1^2 \cdots \mf{r}_k^2)} \xi(\vec{\mf{r}})^2 .
\end{aligned}
\]
Any $\mf{r}$ with $N_{K}(\mf{r}) \leq R^2$ has at most $\frac{\log R}{\log z}  \ll \log N$ prime factors with norm $>z$. Hence, for all $\mf{r}_i$,
\begin{equation}\label{eq7.23}
\frac{\varphi\left(\mf{r}_i\right)}{N_{K}(\mf{r}_i)}=\prod_{\mf{p} \mid \mf{r}_i}(1-1 / N_{K}(\mf{p}))=1+O\left(\frac{\log N}{z}\right).
\end{equation}
Since $|\lambda(\vec{\mf{d}})| \leq 1$, we have that
$$
\xi(\vec{\mf{r}}) \leq\left(\sum_{\substack{N_{K}(\mf{d}) \leq R^2 \\ P^{-}(N_K(\mf{d}))>z}} \frac{1}{d}\right)^k \leq \prod_{z<N_{K}(\mf{p}) \leq R^2}\left(1+\frac{1}{N_{K}(\mf{p})}\right)^k \ll (\log N)^{k+1}.
$$
It follows that 
$$
B^{\prime}=\sum_{\vec{\mf{r}} \in \mathscr{D}} \frac{\xi(\vec{\mf{r}})^2}{N_{K}(\mf{r}_1 \cdots \mf{r}_k)}\left(1+O\left(\frac{\log N}{z}\right)\right)=\sum_{\vec{\mf{r}} \in \mathscr{D}} \frac{\xi(\vec{\mf{r}})^2}{N_{K}(\mf{r}_1 \cdots \mf{r}_k)}+O\left(\frac{\log ^{4k+3} N}{z}\right) .
$$
Since the error term above is $ \ll 1 / \log ^{9998 k^2} N,$ the proposition follows from \eqref{eq7.21}.
\end{proof}

\begin{proposition}\label{prop7.9}
Let $\mu^{+}$be an upper bound sieve function from Theorem \ref{flparta} with parameters $z, D$. Let $\lambda(\vec{\mf{d}})$ satisfy \eqref{eq7.13-gaps} and be supported on $\mathscr{D}$. For $\vec{\mf{r}} \in \mathscr{D}$, define
$$
{\zeta}_1(\vec{\mf{r}})=1_{\mf{r}_1=(1)} \sum_{\substack{\vec{\mf{d}} \in \mathscr{D} \\ \mf{d}_1=(1)}} \frac{\lambda\left(\mf{r}_1 \mf{d}_1, \ldots, \mf{r}_k \mf{d}_k\right)}{N_{K}(\mf{d}_1) \cdots N_{K}(\mf{d}_k)} .
$$

Let $\left(\dc{a}_1 n + \dc{b}_1, \ldots, \dc{a}_k n+\dc{b}_k\right)$ be an admissible set of linear forms, with $k \leq (\log N)^{1/9}$ and $k$ larger than a suitable (absolute) constant, such that $(\dc{a}_1,\dc{b}_1) = (1,0)$,
$$
\begin{aligned}
& 1 \leq N_{K}(\dc{a}_i) \leq N^2, \quad\ N_{K}(\dc{b}_i) \leq N^2 \quad(i \neq 1) .
\end{aligned}
$$

Define $E$, $\mathscr{E}$ by \eqref{edef}, $V$ by \eqref{vdef} and $w(\dc{n})$ by \eqref{wdef}. 
 Then, we have
\begin{align*}
\sum_{N<N_{K}(\dc{n}) \leq 2 N} w(\dc{n}) 1_{n \text{ prime }}&=\frac{V \cdot (|U|/h) \cdot (\Li(2N) - \Li(N))}{\prod_{\substack{\mf{p} \nmid H \\ \rad(N_K(\mf{p})) \leq z}}(1-1 / N_{K}(\mf{p}))} \\
&\qquad \cdot \sum_{\vec{\mf{r}} \in \mathscr{D}} \frac{{\zeta}_m(\vec{\mf{r}})^2}{N_{K}(\mf{r}_1) \cdots N_{K}(\mf{r}_k)}+O\left(\frac{N}{(\log N)^{40 k^2}}\right).
\end{align*}
\end{proposition}
\begin{proof}
Again expanding the square in the definition of $w(\dc{n})$ and interchanging the order of summation, we have that
\begin{equation}\label{eq7.27}
\begin{split}
\sum_{N<N_{K}(\dc{n}) \leq 2 N} w(\dc{n}) 1_{\dc{n} \text { prime }}  &=\sum_{\substack{\rad(N_{K}(\mf{t})) \leq D \\ P^{+}(N_K(\mf{t})) \leq z, (\mf{t}, H) = 1}} \mu^{+}(\rad(N_{K}(\mf{t})))\cdot \frac{\mu(\mf{t})}{\mu(\rad(N_{K}(\mf{t})))}  \\
&\qquad \cdot \sum_{\substack{\vec{\mf{d}}, \vec{\mf{e}} \in \mathscr{D} \cap \mathscr{E}}} \lambda(\vec{\mf{d}}) \lambda(\vec{\mf{e}}) \sum_{\substack{N<N_{K}(\dc{n}) \leq 2 N \\ \mf{t} \mid(\dc{a}_1 n+\dc{b}_1) \cdots(\dc{a}_k n+\dc{b}_k) \\ \forall j:[\mf{d}_j, \mf{e}_j \mid(\dc{a}_j n+\dc{b}_j) \\ \dc{n} \text { prime }}} 1 .
\end{split}
\end{equation}
Since \(N_K(\left[\mf{d}_1, \mf{e}_1\right]) \leq R^4 = N^{\theta - 4 / s} < N / 2\), \(\dc{p} := \dc{n}\) is prime and
\[
N_{K}(\dc{p}) \geq \frac{N}{2},
\]
it follows that
\[
\mf{d}_1 = \mf{e}_1 = 1.
\]
Since \(\mf{d}, \vec{\mf{e}} \in \mathscr{E}\), Lemma \ref{rp} implies that for all $i$, \((\mf{d}_i \mf{e}_i, \dc{a}_i) = 1\), and if $i \neq j$, then \((\mf{d}_i \mf{e}_i, \mf{d}_j \mf{e}_j) = 1\). Therefore, for each \(i \neq 1\), the condition \(\left[\mf{d}_i, \mf{e}_i\right] \mid \dc{a}_i \dc{n} + \dc{b}_i\) is equivalent to
\[
\dc{n} \equiv - \dc{a}_i^{-1} \dc{b}_i \pmod{\left[\mf{d}_i, \mf{e}_i\right]}.
\]
Consequently, $\dc{p}$ lies in a single residue class modulo \(\left[\mf{d}_i, \mf{e}_i\right]\). Moreover, this residue class is coprime to \(\left[\mf{d}_i, \mf{e}_i\right]\), since \(\mf{d}, \vec{\mf{e}} \in \mathscr{E}\). We have \(\mf{t} \mid \prod_{i=1}^k (\dc{a}_i \dc{n} + \dc{b}_i)\) and \((\dc{n}, \mf{t}) = 1\). It follows that
\[
\prod_{i \neq 1} (\dc{a}_i \dc{n} + \dc{b}_i) \equiv 0 \pmod{\mf{t}}, \quad (\dc{n}, \mf{t}) = 1.
\]
This defines \(\rho^*(\mf{t})\) residue classes for \(\dc{p}\) modulo \(\mf{t}\), where $\rho^*(\mf{p}) = \rho(\mf{p})-1$ for prime ideals $\mf{p}$. Therefore, the prime \(\dc{p}\) lies in one of \(\rho^*(\mf{t})\) reduced residue classes modulo \(\mf{t}\). Thus, the inner sum in \eqref{eq7.27} defines exactly \(\rho^*(\mf{t}) \) reduced residue classes for the prime $\dc{p}$ modulo \(\mf{t} \left[\mf{d}_1, \mf{e}_1\right] \cdots \left[\mf{d}_k, \mf{e}_k\right]\). Let \(\mf{u} = \mf{t} \left[\mf{d}_1, \mf{e}_1\right] \cdots \left[\mf{d}_k, \mf{e}_k\right]\), and define \(E(\mf{u})\) by
\[
E(\mf{u}) = \max_{(\mf{u}, s) = 1} \left| \pi(2 N ; \mf{u}, s) - \pi(N ; \mf{u}, s) - \frac{\Li(2 N) - \Li(N)}{h(\mf{u})} \right|,
\]
Then, by \eqref{eq7.27} ,
\begin{align*}
\sum_{N < N_{K}(\dc{n}) \leq 2 N} w(\dc{n}) 1_{\dc{n} \text{ prime}} &= \sum_{\substack{\rad(N_{K}(\mf{t})) \leq D \\ (\mf{t}, H) = 1}} \mu^{+}(\rad(N_{K}(\mf{t}))) \cdot \frac{\mu(\mf{t})}{\mu(\rad(N_{K}(\mf{t})))}\\
& \cdot \sum_{\substack{\vec{\mf{d}}, \vec{\mf{e}} \in \mathscr{D} \cap \mathscr{E} \\ \mf{d}_1 = \mf{e}_1 = 1}} \lambda(\vec{\mf{d}}) \lambda(\vec{\mf{e}}) \left[ \rho^*(\mf{t}) \frac{(\Li(2N)-\Li(N))}{h(\mf{u})} + O\left( \rho(\mf{t}) E(\mf{u}) \right) \right] \\
&= (|U|/h) \cdot (\Li(2N)-\Li(N))V^* B^* + T^*,
\end{align*}
where, since \(P^{+}(N_K(\mf{t})) \leq z < P^{-}(N_K\left( \left[\mf{d}_1, \mf{e}_1\right] \cdots \left[\mf{d}_k, \mf{e}_k\right] \right))\), we have that
\[
\begin{aligned}
V^* & = \sum_{\substack{\rad(N_{K}(\mf{t})) \leq D \\ P^+(N_K(\mf{t})) \leq z, (\mf{t}, H) = 1}} \frac{\mu^{+}(\rad(N_{K}(\mf{t}))) \cdot \frac{\mu(\mf{t})}{\mu(\rad(N_{K}(\mf{t})))} \rho^*(\mf{t})}{\varphi(\mf{t})}, \\
B^* & = \sum_{\substack{\vec{\mf{d}}, \vec{\mf{e}} \in \mathscr{D} \cap \mathscr{E} \\ \mf{d}_1 = \mf{e}_1 = 1}} \frac{\lambda(\vec{\mf{d}}) \lambda(\vec{\mf{e}})}{\varphi\left( \left[\mf{d}_1, \mf{e}_1\right] \cdots \left[\mf{d}_k, \mf{e}_k\right] \right)}, \\
|T^*| & \ll \sum_{\substack{N_{K}(\mf{t}) \leq D^2 \\ P^{+}(N_K(\mf{t})) \leq z}} \rho(\mf{t}) \mu^2(\mf{t}) \sum_{\vec{\mf{d}}, \vec{\mf{e}} \in \mathscr{D} \cap \mathscr{E}} E(\mf{u}).
\end{aligned}
\]

We now utilize Lemma \ref{page-bv} to estimate the error term \(T^*\). Define $x = 2N$. Since \(\vec{\mf{d}}, \vec{\mf{e}} \in \mathscr{D}\), the moduli $\mf{u}$ satisfy
\[
N_{K}(\mf{u}) \leq N_{K}(\mf{t} \mf{d}_1 \cdots \mf{d}_k \mf{e}_1 \cdots \mf{e}_k) \leq D^2 R^4 \leq N^{\theta - \frac{2}{s}} \leq x^\theta
\]

if \(N\) is large enough. For each squarefree \(\mf{q} = \left[ \mf{d}_1, \mf{e}_1 \right] \cdots \left[ \mf{d}_k, \mf{e}_k \right]\), there are \(\leq (3 k)^{\omega(\mf{q})}\) ways to choose \(\mf{d}_1, \mf{e}_1, \ldots, \mf{d}_k, \mf{e}_k\). Also, \(\rho(\mf{t}) \leq k^{\omega(\mf{t})}\). Thus, by Cauchy-Schwarz and the bound

\[
E(\mf{u}) \ll \frac{x}{N_{K}(\mf{u})} \ll \frac{N}{N_{K}(\mf{u})},
\]
we obtain the estimate
\[
\begin{aligned}
|T^*| & \ll \sum_{\substack{N_{K}(\mf{t}) \leq D^2 \\ P^{+}(N_K(\mf{t})) \leq z}} \mu^2(\mf{t}) k^{\omega(\mf{t})} \sum_{\substack{P^{-}(N_K(\mf{q})) > z \\ N_{K}(\mf{q}) \leq R^4}} \mu^2(\mf{q}) (3 k)^{\omega(\mf{q})} E(\mf{t} \mf{q}) \\
& \leq \sum_{N_{K}(\mf{r}) \leq D^2 R^4} \mu^2(\mf{r}) (3 k)^{\omega(\mf{r})} E(\mf{r})^{1/2} \left( \frac{N}{N_{K}(\mf{r})} \right)^{1/2} \\
& \ll \left( N  \right)^{1/2} \left( \sum_{P^{+}(N_K(\mf{r})) \leq N} \frac{\mu^2(\mf{r}) (3 k)^{2 \omega(\mf{r})}}{N_{K}(\mf{r})} \right)^{1/2} \left( \sum_{\mf{r} \leq D^2 R^4	} E(\mf{r}) \right)^{1/2} \\
& \ll \left( N  \right)^{1/2} e^{O(k^2)} (\log N)^{9 k^2 / 2} \left( \frac{x}{(\log N)^{1000 k^2}} \right)^{1/2}.
\end{aligned}
\]
Since $x = 2N$ , we conclude that
\[
T^* \ll \frac{N}{(\log N)^{100 k^2}}.
\]

We now turn to estimating $B^*$. The same argument leading to \eqref{eq7.23} yields that
\[
\prod_{i=1}^k \frac{N_{K}(\left[ \mf{d}_i, \mf{e}_i \right])}{\varphi\left( \left[ \mf{d}_i, \mf{e}_i \right] \right)} = 1 + O\left( \frac{k\log N}{z} \right).
\]
Hence, by the argument in the display following \eqref{eq7.19},
\[
B^* = O\left( \frac{1}{(\log N)^{9998k^2}} \right) + \sum_{\substack{ \vec{\mf{d}}, \vec{\mf{e}} \in \mathscr{D} \cap \mathscr{E} \\ \mf{d}_1 = \mf{e}_1 = 1 }} \frac{\lambda(\vec{\mf{d}}) \lambda(\vec{\mf{e}})}{ N_{K}(\left[ \mf{d}_1, \mf{e}_1 \right] \cdots \left[ \mf{d}_k, \mf{e}_k \right]) }.
\]
As in the proof of Proposition \ref{prop7.8}, the terms with \(\vec{\mf{d}} \notin \mathscr{E}\) or \(\vec{\mf{e}} \notin \mathscr{E}\) contribute \( O\left( \frac{1}{(\log N)^{9998k^2}} \right) \). Using \eqref{eq7.22} and \eqref{eq7.23} again, we obtain
\[
B^* = O\left( \frac{1}{(\log N)^{9998k^2}} \right) + \sum_{\vec{\mf{r}} \in \mathscr{D}} \frac{ {\zeta}_1(\vec{\mf{r}})^2 }{ N_{K}(\mf{r}_1 \cdots \mf{r}_k) },
\]

Finally, we apply the Fundamental Lemma (Theorem \ref{thm3.6}) with the function
\[
g(n) = \frac{1_{(n,H)=1}}{\mu(n)} \sum_{\rad(N_{K}(\mf{t}))=n} \frac{ \rho^*(\mf{t}) \mu(\mf{t}) }{ \varphi\left( \mf{t} \right) }.
\]

We have, for primes $p$ with $(p) = \mf{p}$ (inert),
\begin{align*}
g(p) &= 1_{\substack{p \leq z \\ p \nmid H}}\Big(\frac{\rho(\mf{p})-1}{\varphi(\mf{p})} \Big) , \\
1 - g(p) &= \prod_{\substack{\rad(N_{K}(\mf{p})) = p \leq z \\ p \nmid H}} \left(1 - \frac{\rho(\mf{p})-1}{\varphi(\mf{p})} \right) .
\end{align*}
and for primes $p$ with $(p) = \mf{p}_1 \mf{p}_2$ (split),
\begin{align*}
g(p) &=  1_{\substack{\rad(N_{K}(\mf{p}_1)), \rad(N_{K}(\mf{p}_2)) = p \leq z \\ p \nmid H}} \left(\frac{\rho(\mf{p}_1)-1}{\varphi(\mf{p}_1)} + \frac{\rho(\mf{p}_2)-1}{\varphi(\mf{p}_2)} - \frac{\rho(\mf{p}_1)-1}{\varphi(\mf{p}_1)} \cdot\frac{\rho(\mf{p}_2)-1}{\varphi(\mf{p}_2)} \right), \\
1 - g(p) &= \prod_{\substack{\rad(N_{K}(\mf{p})) = p \leq z \\ p \nmid H}} \left(1 - \frac{\rho(\mf{p})-1}{\varphi(\mf{p})} \right) .
\end{align*}

Observe that \(g(p) \leq \frac{4 k}{N_{K}(p)}\) for all \(p\), thus \((\Omega)\) holds with \(\kappa = 4 k\). Then, by Theorem \ref{flparta},

\begin{small}
\[
V^* = \left( 1 + O\left( e^{- \frac{1}{2} s \log s} \right) \right) \prod_{\substack{ \rad(N_{K}(\mf{p})) \leq z \\ \mf{p} \nmid H}} \left( 1 - \frac{\rho(\mf{p})-1}{\varphi(\mf{p})} \right) = \left( 1 + O\left( \frac{1}{ (\log N)^{9999 k^2} } \right) \right) V^{**},
\]
\end{small}
where
\begin{align*}
V^{**} &= \prod_{\substack{ \rad(N_{K}(\mf{p})) \leq z \\ \mf{p} \nmid H}} \left( 1 - \frac{ \rho(\mf{p})-1 }{ N_{K}(\mf{p}) - 1 } \right) \\
&= \prod_{\substack{ \rad(N_{K}(\mf{p})) \leq z \\ \mf{p} \nmid H}} \left( 1 - \frac{\rho(\mf{p}) }{ N_{K}(\mf{p}) } \right) \prod_{\substack{ \rad(N_{K}(\mf{p})) \leq z \\ \mf{p} \nmid H }} \left( 1 - \frac{ 1 }{ N_{K}(\mf{p}) } \right) ^{-1}	\\
&= V \prod_{\substack{ \rad(N_{K}(\mf{p})) \leq z \\ \mf{p} \nmid H}} \left( 1 - \frac{ 1 }{ N_{K}(\mf{p}) } \right) ^{-1}	.
\end{align*}
Therefore,
\[
V^* = V\prod_{\substack{ \rad(N_{K}(\mf{p})) \leq z \\ \mf{p} \nmid \disc(K) \cdot N_K(\mf{B}) }} \left( 1 - \frac{ 1 }{ N_{K}(\mf{p})} \right)^{-1} + O\left( \frac{1}{ (\log N)^{9995 k^2} } \right).
\]
The same argument leading to \eqref{eq7.19} yields \( B^* \ll (\log N)^{3k+1} \), which completes the proof of the proposition.
\end{proof}

\begin{lemma}\label{inversion1}
	For all $\vec{\mf{r}} \in \mathscr{D}$ and $1 \leq m \leq k$,
$$
{\zeta}_1(\vec{\mf{r}})=1_{\mf{r}_1=(1)} \sum_{\mf{b} \ideal \mc{O}_K} \frac{\mu(\mf{b}) \xi\left(\mf{r}_1, \ldots, \mf{r}_{m-1}, \mf{b}, \mf{r}_{m+1}, \ldots, \mf{r}_k\right)}{N_{K}(\mf{b})} .
$$

\end{lemma}
\begin{proof} Let $\mf{r}_1=(1)$. By \eqref{xidef}, the right side equals
$$
\begin{aligned}
& =\sum_{\mf{b}} \frac{\mu(\mf{b})}{N_{K}(\mf{b})} \sum_{\vec{\mf{d}}} \frac{\lambda\left(\mf{b},\mf{r}_2 \mf{d}_2 \cdots, \mf{r}_k \mf{d}_k\right)}{N_{K}(\mf{d}_1) \cdots N_{K}(\mf{d}_k)} \\
& =\sum_{\mf{d}_i: i \neq 1} \frac{1}{\prod_{i \neq 1} N_{K}(\mf{d}_i)} \sum_{\mf{l} \ideal \mc{O}_K} \frac{\lambda\left(\mf{l}, \mf{r}_2 \mf{d}_2,  \cdots, \mf{r}_k \mf{d}_k\right)}{N_{K}(\mf{l})} \sum_{\mf{b} \mid \mf{l}} \mu(\mf{b})={\zeta}_1(\vec{\mf{r}}).
\end{aligned}
$$
\end{proof}
\begin{lemma}\label{inversion2}
	For all $\mf{d} \ideal \mc{O}_K$,
$$
\lambda(\vec{\mf{d}})=1_{\vec{\mf{d}} \in \mathscr{D}} \sum_{\mathbf{b}} \frac{\mu\left(\mf{b}_1\right) \cdots \mu\left(\mf{b}_k\right) \xi\left(\mf{b}_1 \mf{d}_1, \ldots, \mf{b}_k \mf{d}_k\right)}{N_{K}(\mf{b}_1) \cdots N_{K}(\mf{b}_k)} .
$$

\end{lemma}
\begin{proof}
Let $\vec{\mf{d}} \in \mathscr{D}$. By \eqref{xidef}, the right side is
$$
\begin{aligned}
& =\sum_{\mathbf{b}} \frac{\mu\left(\mf{b}_1\right) \cdots \mu\left(\mf{b}_k\right)}{N_{K}(\mf{b}_1) \cdots N_{K}(\mf{b}_k)} \sum_{\vec{\mf{e}}} \frac{\lambda\left(\mf{b}_1 \mf{d}_1 \mf{e}_1, \ldots, \mf{b}_k \mf{d}_k \mf{e}_k\right)}{N_{K}(\mf{e}_1) \cdots N_{K}(\mf{e}_k)} \\
& =\sum_{\mathbf{l}} \frac{\lambda\left(\mf{l}_1 \mf{d}_1, \ldots, \mf{l}_k \mf{d}_k\right)}{N_{K}(\mf{l}_1) \cdots N_{K}(\mf{l}_k)} \prod_{i=1}^k \sum_{\mf{b}_i \mid \mf{l}_i} \mu\left(\mf{b}_i\right)=\lambda(\vec{\mf{d}}) .
\end{aligned}
$$
\end{proof}

We require the following lemma to estimate sums over rough numbers:
\begin{lemma}\label{lmm:MultipleSummationModified}
Let \( r \leq k \ll (\log R)^{1/5} \). Let \( W_1, \dots, W_r \) be positive integers, each with all prime factors at most $(\log R)^{99999k^2}$, and each a multiple of all primes \( p \leq (\log R)^{4000k^2} \). Let \( g \) and $h$ be arithmetic functions with $g$ multiplicative,  $g(p)/h(p) = 1 + O(k/p)$, $h(p) \gg p$ for $p \geq (\log R)^{4000k^2}$, and for all $x \geq 2$,
\begin{align*}
\sum_{p \leq x} \frac{\log p}{h(p)} = \log(x) + O(1).
\end{align*}
Let \( G: \mathbb{R} \rightarrow \mathbb{R} \) be a smooth function supported on the interval \( [0, 1] \) such that
\[
\sup_{t \in [0, 1]} \left( |G(t)| + |G'(t)| \right) \leq \Omega_G \int_0^\infty G(t) \, dt,
\]
for some quantity \( \Omega_G \) satisfying \( r \Omega_G  = o\left( \dfrac{\log R}{k^2 \log\log R} \right) \).

Let \( \Phi: \mathbb{R} \rightarrow \mathbb{R} \) be smooth with \( \Phi(t), \Phi'(t) \ll 1 \) for all \( t \).

Then for \( k \) sufficiently large, we have
\begin{align*}
\sum_{\substack{\vec{\mf{e}} \in \mathbb{N}^r \\ (e_i, W_i) = 1\ \forall i}} \frac{\mu^2(e)}{g(e)} & \Phi\left( \sum_{i=1}^k \frac{\log e_i}{\log R} \right) \prod_{i=1}^k G\left( \frac{\log e_i}{\log R} \right) = \Pi_g (\log R)^r \idotsint\limits_{t_1, \dots, t_r \geq 0} \Phi\left( \sum_{i=1}^r t_i \right) 	\prod_{i=1}^r G(t_i) \, dt_i \\
&\quad + O\left( r \Omega_G k^2 \log\log R \cdot \Pi_g (\log R)^{r - 1} \idotsint\limits_{t_1, \dots, t_r \geq 0} \prod_{i=1}^r G(t_i) \, dt_i \right),
\end{align*}
where
\[
\Pi_g = \prod_p \left( 1 + \frac{n(p)}{g(p)} \right) \left( 1 - \frac{1}{p} \right)^r, \qquad n(p) = \#\left\{ i \in \{1, \dots, r\} : p \nmid W_i \right\}.
\]
\end{lemma}
\begin{proof}
This lemma is nearly identical to \cite{denseclusters}{Lemma 8.4}. The only change required to the proof of \cite{denseclusters}{Lemma 8.4} is that $L \ll k^2 \log\log R$ rather than $L \ll \log\log R$. 
\end{proof}

Let $\psi:[0,\infty)\rightarrow[0,1]$ be a fixed smooth non-increasing function supported on $[0,1]$ which is $1$ on $[0,9/10]$. Let $F:\mathbb{R}^k\rightarrow\mathbb{R}$ be the smooth function defined by
\begin{equation}\label{eq:FDef}
F(t_1,\dots,t_k)=\psi\Bigl(\sum_{i=1}^kt_i\Bigr)\prod_{i=1}^k\frac{\psi(t_i/U_k)}{1+T_kt_i},\qquad T_k=k\log{k},\qquad U_k=k^{-1/2}.
\end{equation}
 In particular, we note that this choice of $F$ is non-negative, and that the support of $\psi$ implies that
\begin{equation}
\lambda_{\vec{\mf{d}}}=0\quad\text{if $\textstyle d=\prod_{i=1}^kd_i>R$.}
\end{equation}
Let $\tilde{g}(t) = \frac{\psi(t/U_k)}{1 + T_k t}$. Let $\Phi_1(t) = (\psi(t))^2$, and let $\Phi_2(t) = \Big(\int_{-\infty}^\infty\psi(t+u)\tilde{g}(u) du\Big)^2$. Finally, let $G(t) = (\tilde{g}(t))^2$. Note that with these definitions, we have that
\begin{align*}
\int_0^\infty G(t) dt \gg \frac{1}{T_k},	
\end{align*}
and
\begin{align*}
\sup_{t \in [0,1]} |G(t)| + |G'(t)| \ll T_k. 	
\end{align*}
Consequently, the function $G$ satisfies the hypotheses of Lemma \ref{lmm:MultipleSummationModified} with $\Omega_{G} \ll T_k^2$. Furthermore, $\Phi_1, \Phi_1' \ll 1$ trivially. The bound $0 \leq \tilde{g}(t) \leq 1$ implies that $\Phi_2(t) , \Phi_2'(t) \ll 1$. Since $k \ll (\log R)^{1/9}$ by assumption, it follows that 
\begin{equation*}
r\Omega_G k^2 \ll (T_k)^2 k^3 \ll k^5 (\log k)^2 = o \Big(\frac{\log R}{\log \log R}\Big),
\end{equation*}
and hence
\begin{equation*}
	r \Omega_G = o \Big(\frac{\log R}{k^2 \log \log R} \Big)
\end{equation*}
as required by the hypotheses of Lemma \ref{lmm:MultipleSummationModified}. Let 

\begin{align*}
\xi(\vec{\mf{r}})&= F\left(\frac{\log \rad(N_{K}(\mf{r}_1))}{\log R}, \ldots, \frac{\log \rad(N_{K}(\mf{r}_k))}{\log R}\right) \\
&\qquad \cdot 1_{\vec{\mf{r}} \in \mathscr{D}} \mu\left(\mf{r}_1\right) \cdots \mu\left(\mf{r}_k\right) \underbrace{\prod_{z<\rad(N_{K}(\mf{p})) \leq R}(1+k / N_{K}(\mf{p}))^{-1}}_{\text {constant }} .
\end{align*}

Define 
\begin{align*}
g(n) = \left(\sum_{\rad(N_{K}(\mf{t}))= n} \frac{1}{N_{K}(\mf{t})}\right)^{-1}.
\end{align*}
Note that if $(p) =\mf{p}_1 \mf{p}_2$ is split, then
\begin{align*}
g(p) &= \left(\frac{2}{p} + \frac{1}{p^2}\right)^{-1} = \frac{p}{2} \cdot \frac{1}{1 + 1/(2p)} = \frac{p}{2}(1 + O(1/p)).
\end{align*}
and if $(p) = \mf{p}$ is inert, then
\begin{align*}
g(p) &= p^2.	
\end{align*}
For finitely many ramified primes, we have $g(p)= p$. Consequently, if we let $h(p) = p^2$ if $p$ inert, $h(p) = p/2$ if $p$ is split, and $h(p) = p$ if $p$ is ramified, this satisfies the conditions $g(p)/h(p) = 1 + O(k/p)$ and $h(p) \gg p$ in Lemma \ref{lmm:MultipleSummationModified}. Furthermore, by Theorem \ref{cheb}, we have that
\begin{small}
\begin{align*}
\sum_{p \leq x} \frac{\log p}{h(p)} &= O(1) +  2\sum_{\substack{p \leq x \\ 	(p) = \mf{p}_1\mf{p}_2}} \frac{\log p}{p} + \sum_{\substack{p \leq x \\ (p) = \mf{p}}} \frac{\log p}{p^2} = \sum_{N_{K}(\mf{p}) \leq x} \frac{\log(N_{K}(\mf{p}))}{N_{K}(\mf{p})} + O(1) = \log x + O(1).
\end{align*}
\end{small}
Finally, note that by Lemma \ref{inversion2}, we have that
\begin{small}
\begin{align*}
\left|\lambda(\vec{\mf{d}}) \prod_{z<\rad(N_{K}(\mf{p})) \leq R}(1+k /N_{K}(\mf{p}))\right| &\leq \sum_{\vec{\mf{b}} \in \mathscr{D}} \frac{1}{N_{K}(\mf{b}_1) \cdots N_{K}(\mf{b}_k)} \leq \sum_{\substack{P^{-}(N_K(\mf{l}))>z \\ P^{+}(N_K(\mf{l})) \leq R^2}} \frac{\mu^2(\mf{l}) k^{\omega(\mf{l})}}{N_{K}(\mf{l})} \\
&=\prod_{z< \rad(N_{K}(\mf{p})) \leq R}(1+k / N_{K}(\mf{p})),
\end{align*}
\end{small}
i.e., $|\lambda(\vec{\mf{d}})| \leq 1$ for $\vec{\mf{d}} \in \ms{D}$. 

We require the following result, which is \cite{denseclusters}{Lemma 8.6}:
\begin{lemma}[Maynard, Lemma 8.6]\label{lmm:IkJk}
Given a square-integrable function $G:\mathbb{R}^k\rightarrow\mathbb{R}$, let
\[I_k(G)=\int_0^\infty\dotsi\int_0^\infty G^2dt_1\dots dt_k,\qquad J_k(G)=\int_0^\infty \dots \int_0^\infty \Bigl(\int_0^\infty G \, dt_k\Bigr)^2dt_1\dots dt_{k-1}.\]
Let $F$ be as given by \eqref{eq:FDef}. Then
\begin{align*}
\frac{1}{(2k\log{k})^k}\ll I_k(F)\le \frac{1}{(k\log{k})^k},\qquad &\frac{\log{k}}{k}\ll \frac{J_k(F)}{I_k(F)}\ll \frac{\log{k}}{k}.
\end{align*}
\end{lemma}
\begin{proposition}\label{prop7.14}
Let $F$ be given by \eqref{eq:FDef}, with $I(F) , J(F) \gg (2k \log k)^{-k}$. Define $\xi$ by \eqref{xidef}.
\begin{enumerate}[label=(\roman*).]
\item Under the hypotheses of Proposition \ref{prop7.8}, we have
$$
\sum_{N< N_{K}(\dc{n}) \leq 2 N} w(\dc{n}) = V N\left(e^{-\gamma} \frac{\log z}{\log R}\right)^k I(F) \left(1 + O\left(\frac{1}{\log^{1/99}(N)} \right) \right) \quad(k, N \rightarrow \infty),
$$

where

$$
I(F)=\int_{\mathcal{R}_k} F^2(\mathbf{x}) d \mathbf{x} .
$$

\item Under the hypotheses of Proposition \ref{prop7.9}, we have
\begin{align*}
\sum_{N<N_{K}(\dc{n}) \leq 2 N} w(\dc{n}) 1_{n \text { prime }} &=  V N\left(e^{-\gamma} \frac{\log z}{\log R}\right)^k \frac{\theta}{4} c_{K,\mf{B}} \cdot J(F) \left(1 + O\left(\frac{1}{\log^{1/99}(N)} \right) \right) \\
&\qquad(k, N \rightarrow \infty),
\end{align*}
where 

\[
c_{K,\mf{B}} 
=
\frac{|U|}{h}\prod_{\substack{p \mid \disc(K) \cdot N_K(\mf{B}) \\ p \leq z}} \left(1 - \frac{1}{p} \right)
 \cdot \lim_{z \to \infty} \frac{e^{-\gamma}/\log z}{ \prod_{\rad(N_{K}(\mf{p})) \leq z} \left(1 - \frac{1}{N_{K}(\mf{p})}\right)},
\]

and

$$
J(F)=\int_{x_2, \ldots, x_k} \ldots \int\left(\int F(\mathbf{x}) d x_1\right)^2 d x_2 \cdots d x_n.
$$
\end{enumerate}
\end{proposition}
\begin{proof}
Observe that if $(p) = \mf{p}_1 \mf{p}_2$ for $p \in \N$ and $r/p < 1/2$,
\begin{align*}
&1 + r\left(\frac{1}{N_{K}(\mf{p}_1)} + \frac{1}{N_{K}(\mf{p}_2)}  + \frac{1}{N_{K}(\mf{p}_1)N_{K}(\mf{p}_2)} \right)	\\
&= \left(1 + \frac{r}{N_{K}(\mf{p}_1)}\right) \left(1 + \frac{r}{N_{K}(\mf{p}_2)} \right) \left(1+  O\left(\frac{r^2}{N_{K}(\mf{p}_1)^2}\right) \right) .
\end{align*}
Let 
\begin{equation*}
\sigma=\prod_{z<p \leq R}(1+k / p)^{-1} = \left(\frac{\log z}{\log R}\right)^k \left(1 + O\left(\frac{k}{\log z}\right) \right).
\end{equation*}
By Proposition \ref{prop7.8}, the definition of $\xi$, Lemma \ref{lmm:MultipleSummationModified} (with $W_i = \prod_{p \leq z} p$ for $1 \leq i \leq r = k$), the fact that $I(F) \gg (2k \log k)^{-k}$ implies $(\log N)^{-Ck^2} = o(I(F)) $ for any fixed $C>0$, and Theorem \ref{cheb}, we have that
\begin{tiny}
\begin{align*}
&\sum_{N < N_{K}(\dc{n}) \leq 2N} w(\dc{n}) \\
 &= VN \sigma^2 \sum_{\vec{\mf{r}} \in \ms{D}} \frac{1}{N_{K}(\mf{r}_1)\dots N_{K}(\mf{r}_k)} \Phi_1\left(\sum_{i=1}^k \frac{\log \rad(N_{K}(\mf{r}_i))}{\log R} \right) \prod_{i=1}^k G\left(\frac{\log \rad(N_{K}(\mf{r}_i))}{\log R}\right) \\
 &\qquad+ O \left(\frac{N}{(\log N)^{9990k^2}}\right)	\\
 &= VN  \sigma^2 \sum_{\substack{\vec{\mf{e}} \in \N^k \\ (e_i,W_i) = 1}} \left(\prod_{i=1}^k \left(\sum_{\rad(N_{K}(\mf{r}_i)) = e_i} \frac{1}{N_{K}(\mf{r}_i)}\right)\right) \\
 &\qquad \Phi_1\left(\sum_{i=1}^k \frac{\log e_i}{\log R} \right) \prod_{i=1}^k G\left(\frac{\log e_i}{\log R}\right)  + O \left(\frac{N}{(\log N)^{9990k^2}}\right)\\
&=   VN \sigma^2 \prod_{p > z}\left(1 + O\left(\frac{k}{p^2}\right)\right) \lim_{n \to \infty}\left( \prod_{\substack{n > p > z \\ (p) = \mf{p}_1 \mf{p_2}}} \left(1 + \frac{2k}{p} \right)  \prod_{n > p > z}  \left(1 - \frac{1}{p}\right)^{k} \right)\prod_{p \leq z}  \left(1 - \frac{1}{p} \right)^k \cdot (\log R)^k I(F) \\
&\qquad \cdot \left(1 + O\left(\frac{1}{\log^{1/99} N}\right) \right)\\
&= VN  \left(e^{-\gamma} \frac{\log z}{\log R}\right)^{k} I(F) \left(1 + O\left(\frac{1}{\log^{1/99} N}\right) \right).  
 \end{align*}
 \end{tiny}
 Let $x_i=\frac{\log(\rad(N_{K}(\mf{r}_i)))}{\log R}$. By Lemma \ref{inversion1}, 
 \begin{small}
\begin{align*}
&\sum_{\vec{\mf{r}} \in \mathscr{D}} \frac{{\zeta}_1^2(\vec{\mf{r}})}{N_{K}(\mf{r}_1) \cdots N_{K}(\mf{r}_k)}\\
&=\sigma^2 \sum_{\substack{\mf{r}_2, \ldots, \mf{r}_k \\ \mu^2\left(\mf{r}_2 \cdots \mf{r}_k\right)=1 \\ P^{-}(N_K\left(\mf{r}_2 \cdots \mf{r}_k\right))>z}} \frac{1}{N_{K}(\mf{r}_2) \cdots N_{K}(\mf{r}_k)} \\
&\qquad \cdot \left(\sum_{\substack{\mu^2\left(\mf{r}_1\right)=1 \\ \vec{\mf{r}} \in \mathscr{D}}} \frac{1}{N_{K}(\mf{r}_1)} F\left(\frac{ \log(\rad(N_{K}(\mf{r}_1)))}{\log R}, \ldots, \frac{\log \rad(N_{K}(\mf{r}_k))}{\log R}\right)\right)^2 \\
&= \left(\frac{\log z}{\log R}\right)^{2k} \\
&\qquad \cdot \sum_{\substack{\mf{r}_2, \ldots, \mf{r}_k \\ \mu^2\left(\mf{r}_2 \cdots \mf{r}_k\right)=1 \\ P^{-}(N_K\left(\mf{r}_2 \cdots \mf{r}_k\right))>z}} \frac{1}{N_{K}(\mf{r}_2) \cdots N_{K}(\mf{r}_k)}\left(e^{-\gamma} \frac{\log R}{\log z}\right)^2 \\
&\qquad \cdot \prod_{i=2}^k G(x_i) \cdot \Phi_2\left(\sum_{i=2}^k x_i \right) \cdot \left(1 + O\left(\frac{1}{\log^{1/99} N}\right) \right)\\
&= (e^{-\gamma})^2 \left(\frac{\log z}{\log R}\right)^{2k-2} \cdot \left(e^{-\gamma} \frac{\log z}{\log R}\right)^{k-1} J(F) \left(1 + O\left(\frac{1}{\log^{1/99} N}\right) \right)\\
&= (e^{-\gamma})^{2} \left(e^{-\gamma}\frac{\log z}{\log R}\right)^{k-1} J(F) \left(1 + O\left(\frac{1}{\log^{1/99} N}\right) \right) .
\end{align*}
\end{small}
Moreover, by Proposition \ref{prop7.9}, we have that
\begin{align*}
\sum_{N<N_{K}(\dc{n}) \leq 2 N} w(\dc{n}) 1_{\dc{n} \text{ prime }}&=\frac{V \cdot (|U|/h) \cdot (\Li(2N) - \Li(N))}{\prod_{\substack{\mf{p} \nmid H \\ \rad(N_K(\mf{p})) \leq z}}(1-1 / N_{K}(\mf{p}))} \\
&\qquad \cdot \sum_{\vec{\mf{r}} \in \mathscr{D}} \frac{{\zeta}_m(\vec{\mf{r}})^2}{N_{K}(\mf{r}_1) \cdots N_{K}(\mf{r}_k)}+O\left(\frac{N}{(\log N)^{40 k^2}}\right).
\end{align*}
Finally, we have that
\begin{align*}
	& (\Li(2N)-\Li(N)) \prod_{\substack{\mf{p} \nmid H \\ \rad(N_K(\mf{p}))\leq z}}\left(1 - \frac{1}{N_{K}(\mf{p})} \right)^{-1} \\
	&\qquad= N \cdot \frac{\theta}{4} c_{K,\mf{B}} \frac{\log z}{\log R} \left(e^{-\gamma}\right)^{-1} \left(1 + O\left(\frac{1}{\log z}\right) \right),
\end{align*}
which completes the proof of the lemma.
\end{proof}

\begin{theorem}\label{thm8.6}[Existence of a good sieve weight]
Let $k \leq \log^{1/9}(x)$ be a positive integer and $\left(\dc{h}_1, \ldots, \dc{h}_k\right)$ an admissible $k$-tuple of distinct elements of $\mc{O}_K$ with $N_K(\dc{h}_i) \leq 2k^2$. Suppose $x$ and $k$ are larger than a suitable absolute constant, and $y$ is defined by \eqref{ydef}, with $c>0$ fixed. Then, there are quantities $\tau$, $u$ satisfying
\begin{equation}\label{eq8.10}
\tau=x^{o(1)}, \quad u \asymp \log k \quad(x \rightarrow \infty), 
\end{equation}
and a non-negative weight function $w^*(\dc{p}, \dc{n})$ defined on $\mathcal{P} \times  \{\dc{z} \in \mc{O}_K: N_{K}(\dc{z}) \leq y\}$ satisfying:
\begin{itemize}
\item
Uniformly for every $\dc{p} \in \mathcal{P}$, one has
\begin{equation}\label{eq8.11}
\sum_{z \in \mc{O}_K} w^*(\dc{p}, \dc{z}) = \tau \frac{y}{\log ^k x} \left(1 + O\left(\frac{1}{\log^{1/99} y}\right) \right).
\end{equation}
\item
Uniformly for every $\dc{q} \in \mathcal{Q}$ and $i=1, \ldots, k$, one has
\begin{equation}\label{eq8.12}
\sum_{\dc{p} \in \mathcal{P}} w^*\left(\dc{p}, \dc{q}-\dc{h}_i \dc{p}\right) = \tau \frac{u}{k} \frac{x / 2}{\log ^k x} \left(1 + O\left(\frac{1}{\log^{1/99} x}\right) \right) .
\end{equation}
\item Uniformly for all $\dc{p} \in \mathcal{P}$ and $\dc{z} \in \mc{O}_K$,
\begin{equation}\label{eq8.13}
w^*(\dc{p}, \dc{z}) \ll x^{o(1)} \quad(x \rightarrow \infty) .
\end{equation}
\end{itemize}
\end{theorem}
\begin{proof}
	Fix $F$ such that \[M_k(F)=\frac{k J(F)}{I(F)} \asymp \log k,\]
	which exists by Lemma \ref{lmm:IkJk}. Let
\begin{equation*}
s=\log _2 x, \quad R=x^{\frac{\theta}{4}-\frac{3 / 2}{s}}, \quad D=R^{1 / s}, \quad z=(\log x)^{9999k^2}, \quad \tilde{N} = 4y.
\end{equation*}
Define $\xi$ and $\lambda$ by \eqref{xidef} and the first display in Proposition \ref{prop7.9} respectively. Observe that if $k$ and $F$ are fixed, $\lambda$ depends only on $R$ and $z$.
	For $\dc{p} \in \mathcal{P}$ and $\dc{n} \in \mc{O}_K$ satisfying $2y < N_{K}(\dc{n} + \frac{1}{2}(\sqrt{2} + 2)\sqrt{y}) \leq 4y$ we define
\begin{align*}
	 w^*(\dc{p}, \dc{n})&=\left(\sum_{\substack{\mf{t} \mid\left(n+\dc{h}_1 p\right) \cdots\left(n+\dc{h}_k p\right) \\ (\mf{t}, H)=1}} \mu^{+}(\mf{t})\right)\left(\sum_{\substack{\forall j: \mf{d}_j \mid  n+\dc{h}_j p \\ (\mf{d}_j,H) = 1}} \lambda(\vec{\mf{d}})\right)^2 \\
	 &\quad(2y <N_{K}(\dc{n} + (\sqrt{2}+2)\sqrt{y}/2) \leq 4y),
\end{align*}
We now apply Proposition \ref{prop7.14} (i), with $N = 2y$ and with the forms $\dc{m} + (\dc{h}_i \dc{p} - \frac{1}{2}(\sqrt{2}+2)\sqrt{y})$ for $1 \leq i \leq k$, for $\dc{m}$ with $N_{K}(\dc{m}) \in (N,2N]$ (i.e., $\dc{m} = \dc{n} + \frac{\sqrt{2}+2}{2}\sqrt{y}$). For this set of forms, we have
\[\dc{E}=\dc{p}^{k(k-1) / 2} \prod_{i<j}\left(\dc{h}_j-\dc{h}_i\right).\]
All prime factors of $\dc{E}$ have norm either $\ll \log^{2/9}(x)$ or $> x/2 > R$. Consequently, if $\vec{\mf{d}} \in \ms{D}$ and $(\mf{d}_i,H)=1$ for all $i$, then $\vec{\mf{d}} \in \ms{E}$. Thus, with $\dc{a}_i = 1$ and $\dc{b}_i = \dc{h}_i \dc{p} - \frac{1}{2}(\sqrt{2}+2)\sqrt{y}$, we have $w^*(\dc{p},\dc{n}) = w(\dc{n}+\frac{1}{2}(\sqrt{2}+2)\sqrt{y})$, and we have $N_K(\dc{a}_i), N_K(\dc{b}_i) \leq N^2$. Consequently, Proposition \ref{prop7.14} (1) implies that 
\begin{align*}
&\sum_{2y <N_{K}(n + \frac{\sqrt{2}+2}{2}\sqrt{y} ) \leq 4y} w^*(\dc{p}, \dc{n}) \\
&\qquad=\sum_{N<N_{K}(\dc{m}) \leq 2 N} w(\dc{m}) = 2 y V\left(e^{-\gamma} \frac{\log z}{\log R}\right)^k I(F) \left(1 + O\left(\frac{1}{\log^{1/99} y}\right) \right),
\end{align*}
where
\begin{equation*}
	V=\prod_{\substack{\rad(N_{K}(\mf{p})) \leq z \\ \mf{p} \nmid H}}\left(1-\frac{\rho(\mf{p})}{N_{K}(\mf{p})}\right).
\end{equation*}
For primes $\mf{p}$ with $\rad(N_{K}(\mf{p})) \leq z$, since $N_{K}(\dc{p}) > x/2 > z^2 \geq (\rad(N_{K}(\mf{p})))^2 \geq N_K(\mf{p})$, we observe that 
\begin{align*}
	&\rho(\mf{p})=\#\left\{\dc{n} \pmod{\mf{p}}:\left(\dc{n}+\dc{h}_1 \dc{p}\right) \cdots\left(\dc{n}+\dc{h}_k \dc{p}\right) \equiv 0 \pmod{\mf{p}}\right\} \\
	&= \#\left\{n \pmod{\mf{p}}:\left(\dc{n}+\dc{h}_1\right) \cdots\left(n+\dc{h}_k\right) \equiv 0\pmod{\mf{p}}\right\} 
\end{align*}
is independent of $\dc{p}$. This proves \eqref{eq8.11}, with 
\[
\tau=2 V\left(e^{-\gamma}(\log x) \frac{\log z}{\log R}\right)^k I(F)=x^{o(1)}.
\]
Fix a prime $\dc{q} \in \mc{Q} $ and index $i \in \{1,2,\dots, k\}$. Then, since $\dc{q}$ is a prime of norm $ > z$, we have that
\begin{align*}
&\sum_{\dc{p} \in \mathcal{P}} w^*\left(\dc{p}, \dc{q} - \dc{h}_i \dc{p}\right) \\
&= \sum_{x/2 < N_{K}(\dc{n}) \leq x} 1_{\dc{n} \text{ prime}} \Bigg( \sum_{\substack{\mf{t} \mid \prod_j ( \dc{q} + ( \dc{h}_j - \dc{h}_i ) \dc{n} ) \\ (\mf{t}, H) = 1}} \mu^{+}(\mf{t}) \Bigg) \Bigg( \sum_{\substack{ \mf{d}_j \mid \dc{q} + ( \dc{h}_j - \dc{h}_i) \dc{n}\ \forall j  \\ (\mf{d}_j,H) = 1}} \lambda(\vec{\mf{d}}) \Bigg)^2.
\end{align*}
Note that
\[E=\abs{\prod_{j \neq i}\left(\dc{h}_j-\dc{h}_i\right) \prod_{\substack{j_1<j_2 \\ j_1 \neq i, j_2 \neq i}}\left(\dc{h}_{j_1}-\dc{h}_{j_2}\right) \dc{q}},\]
again has all of its prime factors $\mf{s}$ with $\rad(N_{K}(\mf{s})) > x > R$ or $\ll k^2$. Consequently, if $\vec{\mf{d}} \in \ms{D}$ and $(\mf{d}_j, H)=1$, then $\vec{\mf{d}} \in \ms{E}$. Furthermore, the bounds required in the hypotheses of Proposition \ref{prop7.14} (ii) hold. Consequently, Proposition \ref{prop7.14} (ii) implies that
\begin{small}
\begin{align*}
\sum_{x/2<N_{K}(\dc{n}) \leq x} w(\dc{n}) 1_{n \text { prime }} &= \frac{(x/2)}{k} 2V  \left(e^{-\gamma} \frac{\log z}{\log R}\right)^k I(F)\cdot \frac{\theta}{8} c_{K,\mf{B}} M_k(F) \left(1 + O\left(\frac{1}{\log^{1/99} x}\right) \right),
\end{align*}
\end{small}
which proves \eqref{eq8.11} and \eqref{eq8.13} with 
\begin{align*}
u &= \frac{\theta}{8} c_{K,\mf{B}} M_k(F).
\end{align*}
Our assumption that $M_k(F) \asymp \log k$ implies that $u \asymp \log k$. 
\end{proof}
\section{Two-stage random selection}\label{section-two-stage-selection}
Let $k = \log^{1/9}(x)$, with $x$ and $k$ sufficiently large to satisfy the hypotheses of Theorem \ref{thm8.6}. Let $\dc{h}_1, \dots \dc{h}_k$ be a $k$-tuple with $N_K(\dc{h}_i) \leq 2k^2$. Define $s, R, D, z, \tilde{N}$ as in Theorem \ref{thm8.6}, and let $\tau, u$ be the quantities guaranteed by the theorem. Finally, let $x,y,z_0$ be defined as in Section \ref{section-rankin}.

For each prime ideal $\mf{s} \in \mc{S}$, we select the residue class $\mathbf{a}_{\mf{s}} \pmod{\mf{s}}$ uniformly at random from $\mc{O}_K/\mf{s}$. Define $\vec{\mathbf{a}} := (\mathbf{a}_{\mf{s}})_{\mf{s} \in \mc{S}}$.

The set $S(\vec{\mathbf{a}})$ is a random subset of $\mc{O}_K$, with each element surviving with probability 
\begin{equation}\label{sigmadef}
	\sigma:=\prod_{\mf{s} \in \mathcal{S}}\left(1-\frac{1}{N_{K}(\mf{s})}\right)=\prod_{\log ^{20} x<N_{K}(\mf{s}) \leq z_0}\left(1-\frac{1}{N_{K}(\mf{s})}\right) .
\end{equation}

Note that by Theorem \ref{landauprimeideal},
\begin{equation*}
\sigma = \frac{\log \left(\log ^{20} x\right)}{\log z_0}\left(1 + \frac{1}{\log_2^{20}(x)} \right)=\frac{100\left(\log _2 x\right)^2}{\log x \log _3 x} \left(1 + \frac{1}{\log_2^{20}(x)} \right).	
\end{equation*}
and similarly,
\begin{equation*}
\mathbb{E}|\mathcal{Q} \cap S(\overrightarrow{\mathbf{a}})|=\sum_{\dc{q} \in \mathcal{Q}} \mathbb{P}(\dc{q} \in S(\overrightarrow{\mathbf{a}}))=\sigma|\mathcal{Q}| = 100 c \frac{x}{\log x}\log_2(x)	 \left(1 + \frac{1}{\log_2^{20}(x)} \right).
\end{equation*}
The following two results follow in exactly the same manner as the corresponding results (Lemma 6.1 and Corollary 5) in \cite{fgkmt}:
\begin{lemma}\label{lemma8.4}
Let $t \leq \log x$, and let $\dc{n}_1, \ldots, \dc{n}_t$ be distinct elements of $\mc{O}_K$ with norm in the interval $\left[-x^2, x^2\right]$. Then

$$
\mathbb{P}\left(n_1, \ldots, n_t \in S(\overrightarrow{\mathbf{a}})\right)=\left(1+O\left(\frac{1}{\log ^{16} x}\right)\right) \sigma^t
$$
\end{lemma}

\begin{corollary}\label{corollary8.5}
With probability $\geqslant 1-O\left(1 / \log ^8 x\right)$, we have

$$
|\mathcal{Q} \cap S(\overrightarrow{\mathbf{a}})|=\left(1+O\left(\frac{1}{\log ^4 x}\right)\right) \sigma|\mathcal{Q}| = 100 c \frac{x}{\log x}\log_2(x)	 \left(1 + \frac{1}{\log_2^{20}(x)} \right).
$$

\end{corollary}

\section{Probability weights}\label{section-probability-weights}
For each $\dc{p} \in \mc{P}$, let $\tilde{\mathbf{n}}_{\mf{p}}$ denote the random element of $\mc{O}_K$ with probability density
\begin{equation}\label{tildedensity}
	\mathbb{P}\left(\tilde{\mathbf{n}}_{\dc{p}}=\dc{n}\right):=\frac{w^*(\dc{p}, \dc{n})}{\sum_{\dc{n}^{\prime} \in \mc{O}_K} w^*\left(\dc{p}, \dc{n}^{\prime}\right)} \quad(N_{K}(\dc{n}) \leq y)
\end{equation}
Consider
\begin{equation}\label{xpdef}
	X_{\dc{p}}(\vec{\dc{a}}):=\mathbb{P}\left(\tilde{\mathbf{n}}_{\dc{p}}+\dc{h}_i \dc{p} \in S(\vec{\dc{a}}) \text { for all } i=1, \ldots, k\right),
\end{equation}
Let 
\begin{equation}\label{padef}
	\mc{P}(\vec{\dc{a}}) = \{\dc{p} \in \mc{P}: \left|X_{\dc{p}}(\vec{\dc{a}})-\sigma^k\right| \leq \frac{\sigma^k}{\log ^3 x}\} .
\end{equation}
Suppose that we are in the event $\vec{\mathbf{a}} = \vec{\dc{a}}$. If $\dc{p} \in \mc{P} \setminus \mc{P}(\vec{\dc{a}})$, we then set $\mathbf{n}_{\dc{p}}=0$. Otherwise, if $\dc{p} \in \mc{P}(\vec{\dc{a}})$, then we let
\begin{equation}
	Z_{\dc{p}}(\vec{\dc{a}} ; \dc{n}):= \begin{cases}\mathbb{P}\left(\tilde{\mathbf{n}}_{\dc{p}}=\dc{n}\right) & \text { if } \dc{n}+\dc{h}_j \dc{p} \in S(\vec{\dc{a}}) \text { for } j=1, \ldots, k \\ 0 & \text { otherwise }\end{cases}
\end{equation}
and let $\mathbf{n}_{\dc{p}}$ be the random element of $\mc{O}_K$ with conditional probability distribution
\begin{equation}\label{cond}
	\mathbb{P}\left(\mathbf{n}_{\dc{p}}=\dc{n} \mid \overrightarrow{\mathbf{a}}=\vec{\dc{a}}\right):=\frac{Z_{\dc{p}}(\vec{\dc{a}} ; \dc{n})}{X_{\dc{p}}(\vec{\dc{a}})}
\end{equation}
Finally, we define
\begin{equation}\label{epdef}
	\vec{\mathbf{e}}_{p}(\vec{\dc{a}}):=\{\mathbf{n}_{\dc{p}}+\dc{h}_i \dc{p} : 	1 \leq i \leq k\} \cap \mathcal{Q} \cap S(\vec{\dc{a}})
\end{equation}

We require the following result, which is Lemma \cite{fgkmt}{Lemma 6.3}. The same proof 	applies (word-for-word):
\begin{lemma}[Lemma 6.3, \cite{fgkmt}]\label{lemma8.8}
\[\mathbb{E}|\mathcal{P}(\overrightarrow{\mathbf{a}})|=|\mathcal{P}|+O\left(\frac{x}{(\log x)^{11}}\right)=|\mathcal{P}|\left(1+O\left(\frac{1}{\log ^{10} x}\right)\right).\]
\end{lemma}
The main result of this section is the following:
\begin{lemma}\label{smc-2}  
With probability $1-o(1)$, we have
\begin{equation}\label{sumno}
\sigma^{-r} \sum_{i=1}^r \sum_{\dc{p}\in \PP(\vec{\mathbf{a}})} 
Z_{\dc{p}}(\vec{\mathbf{a}};\dc{q}-\dc{h}_i \dc{p}) = \left(1 + O\left(\frac{1}{\log^{3}_2 x}\right)\right) \frac{u}{\sigma} \frac{x}{2y}
\end{equation}
for all but at most $\frac{x}{2\log x \log_2 x}$ of the primes $\dc{q} \in \QQ \cap S(\vec{\mathbf{a}})$.
\end{lemma}
The result above yields the following immediate corollary:
\begin{corollary}\label{immediate}
With probability $1-o(1)$ in $\vec{\mathbf{a}}$, for all but at most $\frac{x}{\log x \log_2 x}$ elements $\dc{q} \in \QQ \cap S(\vec{\mathbf{a}})$, one has
\begin{equation}
\sum_{\dc{p} \in \PP} \PR( \dc{q} \in \mathbf{e}_{\dc{p}}(\vec{\dc{a}})  | \vec{\mathbf{a}}=\vec{\dc{a}}) = \frac{u}{\sigma} \frac{x}{2y} + O_{\leq} \pfrac{1}{(\log_2 x)^2}.
\end{equation}
\end{corollary}
\begin{proof}
From \eqref{sumno}, and observing that $\dc{q}=\mathbf{n}_{\dc{p}}+\dc{h}_i \dc{p}$ is only 
possible if $\dc{p}\in \PP(\vec{\mathbf{a}})$, we find that
\begin{align*}
\sigma^{-r} \sum_{i=1}^r \sum_{\dc{p}\in \PP(\vec{\dc{a}})} 
Z_{\dc{p}}(\vec{\dc{a}};\dc{q}-\dc{h}_ip) &= \sigma^{-r}  \sum_{i=1}^r \sum_{p\in
  \PP(\vec{\dc{a}})} X_{\dc{p}}(\vec{\dc{a}}) \PR(\mathbf{n}_{\dc{p}}=\dc{q}-\dc{h}_ip |
\vec{\ba}=\vec{\dc{a}})\\
&=\left(1+O\pfrac{1}{\log^3 x}\right)  \sum_{i=1}^r \sum_{\dc{p}\in \PP(\vec{\dc{a}})}
 \PR(\mathbf{n}_{\dc{p}}=\dc{q}-\dc{h}_ip | \vec{\ba}=\vec{\dc{a}})\\
&= \left(1+O\pfrac{1}{\log^3 x}\right) \sum_{\dc{p}\in\PP} \PR(\dc{q}\in
\mathbf{e}_{\dc{p}}(\vec{\dc{a}}) |  \vec{\ba}=\vec{\dc{a}}).
\end{align*}
\end{proof}
\begin{proof}[Proof of Lemma \ref{smc-2} ]
By precisely the same argument as in the proof of \cite{fgkmt}{Lemma 6.2}, we have that
\begin{equation}\label{soo-2}
\E \sum_{\dc{n}} \sigma^{-r} \sum_{\dc{p}\in \PP\backslash \PP(\vec{\ba})} Z_{\dc{p}}(\vec{\ba};\dc{n})
=o\left(
\frac{u}{\sigma} \frac{x}{2y}\ \frac{1}{r} \frac{1}{\log_2^3 x}\ \frac{x}{\log x\log_2 x} \right),
\end{equation}
and consequently, it suffices to show that with probability $1-o(1)$,
for all but at most $\frac{x}{4\log x\log_2 x}$ primes $q \in \QQ \cap S(\vec{\mathbf{a}})$, one has
\begin{equation}\label{soo}
\sum_{i=1}^r \sum_{\dc{p}\in \PP} Z_{\dc{p}}(\vec{\ba};\dc{q}-\dc{h}_ip)
 = \left(1 + O_{\leq}\left(\frac{1}{\log^{3}_2 x}\right)\right) 
\sigma^{r-1} u \frac{x}{2y}.
\end{equation}

Observe that by Theorem \ref{thm8.6}, \eqref{eq8.11} and \eqref{eq8.12}, we have that
\begin{align*}
\sum_{\dc{p} \in \mathcal{P}} \mathbb{P}\left(\dc{q}=\tilde{\mathbf{n}}_{\dc{p}}+\dc{h}_i \dc{p}\right)&=\sum_{\dc{p} \in \mathcal{P}}\frac{ w^*\left(\dc{p}, \dc{q}-\dc{h}_i \dc{p}\right)}{\sum_m w^*(\dc{p}, \dc{m})} \\
&= \frac{u}{k} \frac{x}{2 y} \left(1 + O\left(\frac{1}{\log^{1/99}(x)} \right) \right) \quad(\dc{q} \in \mathcal{Q}, 1 \leq i \leq k) .
\end{align*}
Define
\begin{equation}\label{fqa}
F(\dc{q} ; \overrightarrow{\mathbf{a}}):=\sigma^{-k} \sum_{i=1}^k \sum_{\dc{p} \in \mathcal{P}} Z_{\dc{p}}\left(\overrightarrow{\mathbf{a}} ; \dc{q}-\dc{h}_i \dc{p}\right)
\end{equation}
Combining the above with Lemma \ref{lemma8.4} and \eqref{tildedensity}, we find that 
\begin{equation*}
\begin{aligned}
\mathbb{E} \sum_{\dc{q} \in \mathcal{Q} \cap S(\overrightarrow{\mathbf{a}})} F(\dc{q} ; \overrightarrow{\mathbf{a}}) & =\sigma^{-k} \sum_{\dc{q} \in \mathcal{Q}} \sum_{i=1}^k \sum_{\dc{p} \in \mathcal{P}} \mathbb{P}\left(\dc{q}+\left(\dc{h}_j-\dc{h}_i\right) p \in S(\overrightarrow{\mathbf{a}}) \forall j\right) \mathbb{P}\left(\tilde{\mathbf{n}}_{\dc{p}}=\dc{q}-\dc{h}_i \dc{p}\right) \\
& =\left(1+O\left(\frac{1}{\log ^{16} x}\right)\right) \sum_{\dc{q} \in \mathcal{Q}} \sum_{i=1}^k \sum_{\dc{p} \in \mathcal{P}} \mathbb{P}\left(\tilde{\mathbf{n}}_{\dc{p}}=\dc{q}-\dc{h}_i \dc{p}\right) \\
&= \left(1 + O\left(\frac{1}{\log^{1/99}(x)} \right)\right) \sum_{\dc{q} \in \mathcal{Q}} \sum_{i=1}^k \frac{u x}{2 k y} \\
&= \left(1 + O\left(\frac{1}{\log^{1/99}(x)} \right)\right)\frac{\sigma y}{\log x}\left(\frac{u x}{2 \sigma y}\right) .
\end{aligned}
\end{equation*}
Similarly, we find that
\begin{small}
\begin{equation*}
\begin{gathered}
\mathbb{E} \sum_{\dc{q} \in \mathcal{Q} \cap S(\overrightarrow{\mathbf{a}})} F(\dc{q} ; \overrightarrow{\mathbf{a}})^2=\sigma^{-2 k} \sum_{\dc{q} \in \mathcal{Q}} \sum_{\dc{p}_1, \dc{p}_2 \in \mathcal{P}} \sum_{i_1, i_2} \mathbb{P}\left(\dc{q}+\left(\dc{h}_j-\dc{h}_{i_{\ell}}\right) p_{\ell} \in S(\overrightarrow{\mathbf{a}}) \text { for } j=1, \ldots, k ;~\ell=1,2\right) \\
\times \mathbb{P}\left(\tilde{\mathbf{n}}_{\dc{p}_1}=\dc{q}-\dc{h}_{i_1} \dc{p}_1\right) \mathbb{P}\left(\tilde{\mathbf{n}}_{\dc{p}_2}=\dc{q}-\dc{h}_{i_2} \dc{p}_2\right).
\end{gathered}
\end{equation*}
\end{small}
Since we have $\PR(\tilde{\mathbf{n}}=n) \ll x^{-0.99}$ the ``diagonal" terms with $p_1 = p_2$ contribute 
\begin{equation*}
\ll \sigma^{-2 k}|\mathcal{Q}| \cdot|\mathcal{P}| k^2\left(x^{-0.99}\right)^2 \ll x^{0.03} .
\end{equation*}
For $\dc{p}_1 \neq \dc{p}_2$ and $\dc{q} \in \mc{Q}$ there are $2k-1$ distinct algebraic integers $\dc{q}+\left(\dc{h}_j-\dc{h}_{i_{\ell}}\right) \dc{p}_{\ell}, 1 \leq j \leq k, 1 \leq \ell \leq 2$, since only the terms $j=i_1$, $\ell=1$, and $j=i_2$, $\ell=2$ are equal. Consequently, by Lemma \ref{lemma8.4}, 
\[
\mathbb{E} \sum_{\dc{q} \in \mathcal{Q} \cap S(\overrightarrow{\mathbf{a}})} F(\dc{q} ; \overrightarrow{\mathbf{a}})^2 = \frac{\sigma y}{\log x}\left(\frac{u x}{2 \sigma y}\right)^2\left(1 + O\left(\frac{1}{\log^{1/99}(x)}\right)\right).
\]
Combining the first and second moment calculations, we find that
\[\begin{aligned} 
&\mathbb{E} \sum_{\dc{q} \in \mathcal{Q} \cap S(\overrightarrow{\mathbf{a}})}\left(F(\dc{q} ; \overrightarrow{\mathbf{a}})-\frac{x u}{2 \sigma y}\right)^2 \\
& =\mathbb{E} \sum_{\dc{q} \in \mathcal{Q} \cap S(\overrightarrow{\mathbf{a}})} F(\dc{q} ; \overrightarrow{\mathbf{a}})^2-2 \frac{x u}{2 \sigma y} \mathbb{E} \sum_{\dc{q} \in \mathcal{Q} \cap S(\overrightarrow{\mathbf{a}})} F(\dc{q} ; \overrightarrow{\mathbf{a}})+\left(\frac{x u}{2 \sigma y}\right)^2 \mathbb{E}|\mathcal{Q} \cap S(\overrightarrow{\mathbf{a}})| \\ & =O\left(\frac{\sigma y}{\log x}\left(\frac{x u}{2 \sigma y}\right)^2 \left(\frac{1}{\log_2^{20}(x)}\right)\right) .
\end{aligned}\]
By Markov's inequality, it follows that the LHS is $\leq \frac{\sigma y}{\log x}\left(\frac{x u}{2 \sigma y}\right)^2 \left(\frac{1}{\log_2^{10}(x)}\right) $ with probability $1 - O\left(\frac{1}{\log_2^{9}(x)}\right)$. In this event, $F(\dc{q} ; \overrightarrow{\mathbf{a}}) = \frac{x u}{2 \sigma y} \left(1 + O_\leq\left(\frac{1}{\log_2^{3} x}\right) \right)$ for all but $O\left(\frac{\sigma y}{\log x} \cdot \frac{1}{\log_2^{3}(x)} \right)$ primes $\dc{q} \in \mathcal{Q} \cap S(\overrightarrow{\mathbf{a}})$. Since $\sigma y / \log x = 100 c \frac{x}{\log x}\log_2(x)	 \left(1 + O\left(\frac{1}{\log_2^{20}(x)} \right)\right)$, the lemma follows.
\end{proof}
\section{Applying the covering theorem}\label{section-applying-covering}
We require the following result, which is a consequence of the hypergraph covering theorem proven in \cite{fgkmt}:
\begin{corollary}[Corollary 4, \cite{fgkmt}]\label{packing-quant-cor}  Let $x\to\infty$.
Let $\PP'$, $\QQ'$ be sets 
with $\# \PP' \le x$ and $\#\QQ' > (\log_2 x)^3$.
For each $\dc{p} \in \PP'$, let $\vec{\mathbf{e}}_{\dc{p}}$ be a random subset of $\QQ'$
satisfying the size bound
\be\label{rbound}
\# \vec{\mathbf{e}}_{\dc{p}} \le r = O\left( \frac{\log x \log_3 x}{\log_2^2 x} \right) \qquad
(\dc{p}\in \PP').
\ee
Assume the following:
\begin{itemize}
\item (Sparsity) For all $\dc{p} \in \PP'$ and $\dc{q} \in \QQ'$,
\begin{equation}\label{qform-quant-cor}
\PR( \dc{q} \in \vec{\mathbf{e}}_{\dc{p}} ) \le x^{-1/2 - 1/10}.
\end{equation}
\item (Uniform covering) For all but at most $\frac{1}{(\log_2 x)^2} \# \QQ'$ elements $\dc{q} \in \QQ'$, we have
\begin{equation}\label{pje-size-bite-cor}
\sum_{\dc{p} \in \PP'} \PR( \dc{q} \in \vec{\mathbf{e}}_{\dc{p}}) = C + O_{\le}\pfrac{1}{(\log_2 x)^2}
 \end{equation}
for some quantity $C$, independent of $\dc{q}$, satisfying
\begin{equation}\label{sigma}
\frac{5}{4} \log 5 \le C \ll 1.
\end{equation}
\item (Small codegrees) For any distinct $\dc{q}_1,\dc{q}_2 \in \QQ'$,
\be\label{small-codegree-cor}
\sum_{\dc{p}\in\PP'} \PR(\dc{q}_1,\dc{q}_2\in \vec{\mathbf{e}}_{\dc{p}}) \le x^{-1/20}.
\ee
\end{itemize}
Then for any positive integer $m$ with
\begin{equation}\label{moo}
 m \le \frac{\log_3 x}{\log 5},
\end{equation}
we can find random sets $\vec{\mathbf{e}}'_{\dc{p}}\subseteq \QQ'$ for each $\dc{p} \in \PP'$ such that
\[
\# \{ \dc{q} \in \QQ':  \dc{q} \not\in \vec{\mathbf{e}}'_{\dc{p}} \hbox{ for all } \dc{p} \in \PP' \} \sim 5^{-m} \# \QQ'
\]
with probability $1-o(1)$.  More generally, for any $\QQ'' \subset \QQ'$ with cardinality at least $(\# \QQ')/\sqrt{\log_2 x}$, one has
\[
\# \{ \dc{q} \in \QQ'':  \dc{q} \not\in \vec{\mathbf{e}}'_{\dc{p}} \hbox{ for all } \dc{p} \in \PP' \} \sim 5^{-m} \# \QQ''
\]
with probability $1-o(1)$.  The decay rates in the $o(1)$ and $\sim$ notation
are uniform in $\PP'$, $\QQ'$, $\QQ''$.
\end{corollary}

In order to prove Theorem \ref{sieved}, we first show the following:
\begin{theorem}[Random construction]\label{sieve-primes-2}  Let $x$ be a sufficiently large real number and define $y$ by \eqref{ydef}.
 Then there is a quantity $C$ with
\begin{equation}\label{sigma-order}
C \asymp \frac{1}{c}
\end{equation}
with the implied constants independent of $c$, a tuple of positive integers
$(\dc{h}_1,\ldots,\dc{h}_k)$ with $k\le \sqrt{\log x}$,
and some way to choose random vectors $\vec{\mathbf{a}}=(\mathbf{a}_{\mf{s}} \pmod{\mf{s}})_{\mf{s} \in \cS}$
and $\vec{\mathbf{n}}=(\mathbf{n}_{\dc{p}})_{\dc{p} \in \PP}$ of congruence classes $\mathbf{a}_s \pmod{s}$ and algebraic integers $\mathbf{n}_{\dc{p}} \in \mc{O}_K$ respectively, obeying the following:
\begin{itemize}
\item For every $\vec{\dc{a}}$ in the essential range of $\vec{\mathbf{a}}$, one has
$$ \PR( \dc{q} \in \mathbf{e}_{\dc{p}}(\vec{\dc{a}})  | \vec{\mathbf{a}} = \vec{\dc{a}}) 
\leq x^{-1/2 - 1/10} \quad (\dc{p}\in \PP),
$$
where $\mathbf{e}_{\dc{p}}(\vec{\dc{a}}):=\{ \mathbf{n}_{\dc{p}}+\dc{h}_i \dc{p} : 1\le i\le r\} \cap \QQ \cap S(\vec{\dc{a}})$.
\item With probability $1-o(1)$ we have that	
\begin{equation}\label{treat}
 \# (\QQ \cap S(\vec{\mathbf{a}})) \sim 100 c \frac{x}{\log x} \log_2 x.
\end{equation}
\item Call an element $\vec{\dc{a}}$ in the essential range of $\vec{\mathbf{a}}$
\emph{good} if, for all but at most 
$\frac{x}{\log x \log_2 x}$ elements $q \in \QQ \cap S(\vec{\dc{a}})$, one has
\begin{equation}\label{good}
\sum_{\dc{p} \in \PP} \PR( \dc{q} \in \mathbf{e}_{\dc{p}}(\vec{\dc{a}})  | \vec{\mathbf{a}}=\vec{\dc{a}}) = C + O_{\le} \pfrac{1}{(\log_2 x)^2}.
\end{equation}
Then  $\vec{\mathbf{a}}$ is good with probability $1-o(1)$.
\end{itemize}
\end{theorem}
\begin{proof}
Let $C := \frac{ux}{2\sigma y}$; note that $C \asymp \frac{1}{c}$. First, observe that if $\dc{p} \in \mc{P} \setminus \mc{P}(\vec{\dc{a}})$, then $\PR(\dc{q} \in \mathbf{e}_{\dc{p}}(\vec{\dc{a}}) | \vec{\mathbf{a}} = \vec{\dc{a}}) = 0$. Otherwise, using the fact that $\PR(\tilde{\mathbf{n}} = \dc{n}) \ll x^{-0.99}$, we find that
	\begin{align*}
	\PR( \dc{q} \in \mathbf{e}_{\dc{p}}(\vec{\dc{a}})  | \vec{\mathbf{a}} = \vec{\dc{a}}) &= \sum_{i=1}^k  \PR(\mathbf{n}_{\dc{p}} = \dc{q} - \dc{h}_i \dc{p}  | \vec{\mathbf{a}} = \vec{\dc{a}}) = \sum_{i=1}^k \frac{Z_{\dc{p}}(\vec{\dc{a}},\dc{q}-\dc{h}_i \dc{p})}{X_{\dc{p}}(\vec{\dc{a}})} \\
	&\ll (\log x)^{1/2} \cdot x^{-0.99}  \cdot \sigma^{-2k} \ll x^{-0.99} \cdot \exp((\log x)^{0.51}) \leq x^{-1/2-1/10}. 
	\end{align*}
	This proves the first assertion of the theorem. The second assertion follows from Corollary \ref{corollary8.5}. Finally, the third assertion follows from Corollary \ref{immediate}.
\end{proof}

Finally, we show how Theorem \ref{sieved} follows from Theorem \ref{sieve-primes-2} and Corollary \ref{packing-quant-cor}:
\begin{proof}[Proof of Theorem \ref{sieved} using Theorem \ref{sieve-primes-2} and Theorem \ref{packing-quant-cor}]
	By \eqref{sigma-order}, if we choose $0 < c < 1/2$ sufficiently small, we can ensure that \eqref{sigma} holds. Take
	\[
 		m = \left\lfloor \frac{\log_3 x}{\log 5} \right\rfloor.
	\]
	Let $\vec{\mathbf{a}}$ and $\vec{\mathbf{n}}$ be the random vectors guaranteed by Theorem \ref{sieve-primes-2}. By Theorem \ref{sieve-primes-2}, there exists some $\vec{\dc{a}}$ such that $\vec{\dc{a}}$ is good and \eqref{treat} holds. 	We intend to apply Corollary \ref{packing-quant-cor} with $\mc{P}' = \mc{P}$ and $\QQ' = \QQ \cap S(\vec{\dc{a}}) $ to the random variables $\mathbf{n}_{\dc{p}}$ conditioned on $\vec{\mathbf{a}} = \vec{\dc{a}}$. 
	
	We now verify that each hypothesis of the Corollary \ref{packing-quant-cor} holds.
	First, note that \eqref{pje-size-bite-cor} follows from \eqref{good}. Similarly, \eqref{qform-quant-cor} follows from the first assertion of Theorem \ref{sieve-primes-2}. Finally, we must verify \eqref{small-codegree-cor}. For distinct $\dc{q}_1, \dc{q}_2 \in \QQ$, observe that if $\dc{q}_1, \dc{q}_2 \in \mathbf{e}_{\dc{p}}(\vec{\dc{a}})$, then $\dc{p} \mid \dc{q}_1 - \dc{q}_2$. However, $\dc{q}_1 - \dc{q}_2$ is a nonzero algebraic integer of norm $O(x \log x)$, and can therefore be divisible by at most one prime $\dc{p}_0 \in \PP'$. Hence,
		\[
			\sum_{\dc{p}\in\PP'} \PR(\dc{q}_1,\dc{q}_2\in\mathbf{e}_{\dc{p}}(\vec{\dc{a}})) \leq  \PR(\dc{q}_1,\dc{q}_2\in\vec{\mathbf{e}}_{\dc{p}_0}(\vec{\dc{a}})) \le x^{-1/2-1/10}.
		\]
		By Corollary \ref{packing-quant-cor} and \eqref{treat}, there exist random variables $\vec{\mathbf{e}}'_{\dc{p}}(\vec{\dc{a}})$ with essential range contained in the essential range of $\mathbf{e}_{\dc{p}}(\vec{\dc{a}}) \cup \{\emptyset\}$, satisfying
		\[
 			\{ \dc{q}\in \QQ \cap S(\vec{\dc{a}}) : \dc{q}\not\in \vec{\mathbf{e}}'_{\dc{p}}(\vec{\dc{a}}) \text{ for all }\dc{p}\in\PP\} 
			\sim 5^{-m} \# (\QQ\cap S(\vec{\dc{a}})) \ll \frac{x}{\log x}
		\]
		with probability $1-o(1)$. Since $\vec{\mathbf{e}}'_{\dc{p}}(\vec{\dc{a}})=\{\mathbf{n}'_{\dc{p}}+\dc{h}_i \dc{p} : 1\le i\le r\} \cap \QQ \cap S(\vec{\dc{a}})$ for some random algebraic integer $\mathbf{n}'_{\dc{p}}$, it follows that
		\[
 			\{ \dc{q}\in \QQ \cap S(\vec{\dc{a}}) : \dc{q}\not \equiv \mathbf{n}'_{\dc{p}}\pmod{p} \text{ for all }\dc{p}\in\PP\} \ll \frac{x}{\log x}
		\]
with probability $1-o(1)$.  Taking a specific $\vec{\mathbf{n}}'=\vec{\dc{n}}'$ for which the above holds and setting $\dc{b}_{\dc{p}}=\dc{n}'_{\dc{p}}$ for all $\dc{p}$ yields the conclusion of Theorem \ref{sieved}.
\end{proof}

\bibliographystyle{abbrv}
\bibliography{mybib}
\end{document}